\documentclass[a4paper,11pt]{amsart}

\newfont{\cyr}{wncyr10 scaled 1100}

\usepackage[left=2.7cm,right=2.7cm,top=3.5cm,bottom=3cm]{geometry}
\usepackage{amsthm,amssymb,amsmath,amsfonts,mathrsfs,amscd,yfonts}
\usepackage{verbatim}
\usepackage{turnstile}
\usepackage[latin1]{inputenc}
\usepackage[all]{xy}
\usepackage{latexsym}
\usepackage{stmaryrd}
\usepackage{dsfont}
\usepackage{longtable}
\usepackage{multirow}
\usepackage[usenames,dvipsnames]{color}
\usepackage{listings}
\usepackage{mathtools}
\usepackage{tikz}
\usepackage{tikz-cd}
\usetikzlibrary{matrix,arrows,decorations.pathmorphing}
\usepackage{bm}
\usepackage[shortlabels]{enumitem}

\theoremstyle{plain}
\newtheorem*{thmA}{Theorem~A}
\newtheorem*{thmB}{Theorem~B}
\newtheorem*{thmC}{Theorem~C}
\newtheorem*{thmD}{Theorem~D}

\newtheorem{theorem}{Theorem}[section]
\newtheorem{corollary}[theorem]{Corollary}
\newtheorem{lemma}[theorem]{Lemma}
\newtheorem{proposition}[theorem]{Proposition}
\newtheorem{propo}[theorem]{Proposition}

\theoremstyle{definition}
\newtheorem{definition}[theorem]{Definition}

\newtheorem{examplewr}[theorem]{Example}

\theoremstyle{remark}
\newtheorem{obswr}[theorem]{Observation}
\newtheorem{remarkwr}[theorem]{Remark}
\newtheorem{hypotheses}[theorem]{Hypotheses}

\newenvironment{remark}{\begin{remarkwr}\begin{upshape}}{\end{upshape}\end{remarkwr}}

\newcommand{\any}{?}

\newcommand{\bb}{\mathbb}
\newcommand{\frk}{\mathfrak}
\newcommand{\cl}{\mathcal}

\newcommand{\univ}{\kappa}
\newcommand{\CM}{{\boldsymbol{\theta}_{\psi_0}}}

\newcommand{\kapinfty}{{\kappa_{\psi,\mathrm{ad}(g),\infty}}}
\newcommand{\kapinftyone}{{\kappa_{\psi,\mathrm{ad}(g),1,\infty}}}
\newcommand{\kapinftyn}{{\kappa_{\psi,\mathrm{ad}(g),m,\infty}}}

\newcommand{\kaps}{{\kappa_{\psi,\mathrm{ad}^0(g)}}}

\newcommand{\kapinftys}{{\kappa_{\psi,\mathrm{ad}^0(g),\infty}}}
\newcommand{\kapinftyones}{{\kappa_{\psi,\mathrm{ad}^0(g),1,\infty}}}
\newcommand{\kapinftyns}{{\kappa_{\psi,\mathrm{ad}^0(g),m,\infty}}}
\newcommand{\kapinftynqs}{{\kappa_{\psi,\mathrm{ad}^0(g),mq,\infty}}}

\newcommand{\Vrep}{V^\psi_{\mathrm{ad}(g)}}
\newcommand{\Trep}{T^\psi_{\mathrm{ad}(g)}}
\newcommand{\Arep}{A^\psi_{\mathrm{ad}(g)}}

\newcommand{\Tsrep}{T^\psi_{\mathrm{ad}^0(g)}}

\DeclareMathOperator{\alt}{Alt}

\DeclareMathOperator{\et}{et}

\DeclareMathOperator{\cyc}{cyc}

\DeclareMathOperator{\nr}{nr}

\DeclareMathOperator{\Ind}{Ind}

\DeclareMathOperator{\lcm}{lcm}
\DeclareMathOperator{\bal}{bal}

\DeclareMathOperator{\Ad}{Ad}

\DeclareMathOperator{\ad}{ad}

\DeclareMathOperator{\Gr}{bal}

\DeclareMathOperator{\Iw}{Iw}

\DeclareMathOperator{\Char}{Char}
\DeclareMathOperator{\tors}{tors}
\DeclareMathOperator{\Frac}{Frac}
\DeclareMathOperator{\ac}{ac}

\newcommand{\Lp}{{\mathscr{L}_p}}

\newcommand{\Q}{\mathbb{Q}}
\newcommand{\Z}{\mathbb{Z}}

\newcommand{\Sel}{\mathrm{Sel}}

\newcommand{\Gal}{\mathrm{Gal\,}}
\newcommand{\GL}{\mathrm{GL}}

\newcommand{\Frob}{\mathrm{Fr}}

\newcommand{\Fr}{\mathrm{Fr}}

\newcommand{\ord}{{\mathrm{ord}}}
\newfont{\gotip}{eufb10 at 12pt}

\newcommand{\cO}{{\mathcal O}}

\newcommand{\hf}{{\mathbf{f}}}
\newcommand{\hg}{{\mathbf{g}}}
\newcommand{\hh}{{\mathbf h}}

\DeclareMathOperator{\Hom}{Hom}

\newcommand{\res}{\mathrm{res}}

\newcommand{\fp}{{\mathfrak p}}

\newcommand{\dBr}[1]{\llbracket{#1}\rrbracket}

\numberwithin{equation}{section}

\include{thebibliography}

\begin{document}

\title{An anticyclotomic Euler system for adjoint modular Galois representations}

\author{Ra\'ul Alonso, Francesc Castella, and \'Oscar Rivero}

\begin{abstract}
Let $K$ be an imaginary quadratic field and $p$ a prime split in $K$. In this paper we construct an anticyclotomic Euler system for the adjoint representation attached to elliptic modular forms base changed to $K$. We also relate our Euler system to a $p$-adic $L$-function deduced from the construction by Eischen--Wan and  Eischen--Harris--Li--Skinner of $p$-adic $L$-functions for unitary groups.
%We also introduce a $p$-adic $L$-function in this context, using a suitable transfer to a unitary group, and show that the base class of our system .
This allows us to derive new cases of the Bloch--Kato conjecture in rank zero, and a divisibility towards an Iwasawa main conjecture.
\end{abstract}

%\date{\today\;{\color{blue} Draft still in progress}}
\date{\today}

\address{R. A.: Department of Mathematics, Princeton University, Fine Hall, Princeton, NJ 08544-1000, USA}
\email{raular@math.princeton.edu}

\address{F. C.: Department of Mathematics, University of California, Santa Barbara, CA 93106, USA}
\email{castella@ucsb.edu}

\address{O. R.: Simons Laufer Mathematical Sciences Institute, 17 Gauss Way, Berkeley, CA 94720, United States of America}
\email{riverosalgado@gmail.com}

\subjclass[2010]{11R23; 11F85, 14G35}

\maketitle

\setcounter{tocdepth}{1}
\tableofcontents

\section{Introduction}

%In this paper we construct an anticyclomic Euler system for ce

%Let $K$ be an imaginary quadratic field

The goal of this paper is to study the Bloch--Kato conjecture and the anticyclotomic Iwasawa theory of certain twists of the adjoint Galois representation attached to elliptic modular forms base changed to an imaginary quadratic field.

Our main result is the construction of an anticyclotomic Euler system in this setting, which we relate to an analogue of the Hida--Schmidt $p$-adic $L$-function for the symmetric square. By Kolyvagin's method for anticyclotomic split Euler systems, as developed by Jetchev--Nekov\'{a}\v{r}--Skinner, our results yield new cases of the Bloch--Kato conjecture in rank zero and a divisibility towards an Iwasawa main conjecture.

\subsection{The set-up}

Let $g\in S_l(N_g,\chi_g)$ be an ordinary newform of weight $l \geq 2$, level $N_g$, and nebentypus $\chi_g$. Let $K/\Q$ be an imaginary quadratic field, and let $\psi$ be a Hecke character of $K$ of infinity type $(1-k,0)$ for some even integer $k\geq 2$. We assume that the associated theta series $\theta_\psi\in S_k(N_\psi)$ has trivial nebentypus. Fix an odd prime $p\nmid 2N_gN_\psi$ and an embedding $\iota_p:\overline{\Q}\hookrightarrow\overline{\Q}_p$, and for simplicity in this Introduction assume that the Hecke field of $g$ and the values of $\psi$ are contained in a number field $L$ with a prime $\mathfrak{P}$ above $p$ such that $L_\mathfrak{P}=\Q_p$. We assume that $p$ splits in $K$ and is a prime of ordinary reduction for $g$ and, again for simplicity, that $p\nmid h_K$, the class number of $K$. We will also assume that $g$ is not of CM-type.

Let $V_g$ be the (dual to Deligne's) $p$-adic Galois representation attached to $g$, and denote by ${\rm ad}^0(V_g)\subset{\rm End}_{\Q_p}(V_g)$ the adjoint representation on the trace-zero endomorphisms of $V_g$. We consider the conjugate self-dual $G_K$-representation
\[
V:={\rm ad}^0(V_g)(\psi_{}^{-1})(1-k/2),
\]
%where ${\rm ad}^0(V_g)\subset{\rm End}_{\Q_p}(V_g)$ is the adjoint representation on the trace-zero endomorphisms of $V_g$,
where $(\psi^{-1})$ denotes the twist by the inverse of $\psi$ and $(1-k/2)$ is the twist by the $(1-k/2)$-th power of the $p$-adic cyclotomic character.
%and $\psi_\frk{P}$ is a $\frk{P}$-adic avatar of $\psi$.

\subsection{Euler systems and $p$-adic $L$-functions}

In this paper we construct an anticyclotomic Euler system for $V$ and relate it to an associated anticyclotomic $p$-adic $L$-function.

For a positive integer $m$ we write $K[m]$ for the maximal $p$-extension inside the ring class field of $K$ of conductor $m$. Denote by $\mathcal{S}'$ the set of all squarefree products of primes $q$ in the positive density set $\mathcal{P}'$ of Definition~\ref{def:P'}; in particular, these primes split in $K$. For any $p$-adic $G_K$-representation $W$ and a prime $\frk{q}$ of $K$, put
\[
P_{\frk{q}}(W;X)=\det(1-\Fr_{\frk{q}}^{-1}X\vert W^\vee(1)),
\]
where $\Fr_\frk{q}$ denotes an arithmetic Frobenius element for the prime $\frk{q}$ and $W^{\vee}$ denotes the contragredient representation of $W$. A natural lattice $T_g\subset V_g$ described in $\S\ref{subsec:Galrep}$ defines a lattice in $V$ denoted by $T$. Finally, let $H^1_{\rm Iw}(K[mp^\infty],T) = \varprojlim_rH^1(K[mp^r],T)$.

\begin{thmA}[Theorem~\ref{thm:ES-T}]\label{thm:ES:intro}
%Suppose that:
%\begin{itemize}
%\item $p$ splits in $K$,
%\item $p$ does not divide the class number of $K$.
%\end{itemize}
Assume that $H^1(K[mp^s],T)$ is torsion-free for all $m \in\mathcal{S}'$ and $s\geq 0$. There exists a collection of classes
\[
\left\lbrace\kapinftyns\in H^1_{\rm Iw}(K[mp^\infty],T)\;\colon\; m\in\mathcal{S}'\right\rbrace
\]
such that whenever $m, mq\in\mathcal{S}'$ with $q$ a prime, we have
\begin{equation}
{\rm cor}_{K[mq]/K[m]}(\kapinftynqs)=P_{\frk{q}}(V;{\rm Fr}_{\frk{q}}^{-1})\,\kapinftyns,\nonumber
\end{equation}
where $\mathfrak{q}$ is any of the primes of $K$ above $q$.
\end{thmA}

We obtain the Euler system classes $\kapinftyns$ from a suitable modification of the diagonal Euler system classes $\kappa_{\psi,g,g^*,m,\infty}$ for
\[
\Vrep:=V_g\otimes V_{g^*}(\psi_{}^{-1})(1-c)
%\simeq V\oplus V'
\]
constructed in \cite{ACR},
%and attached to the triple $(\theta_\psi,g,g^*)$,
where $g^*=g\otimes\chi_{g}^{-1}$ is the twist of $g$ by the inverse of its nebentypus, and $c=(k+2l-2)/2$.
%and $V'=\Q_p(\psi^{-1})(1-k/2)$.
It follows from our construction (and the results of \cite{BSV} that it builds upon) that $\kapinftyns$ lands in the balanced Selmer groups ${\rm Sel}_{\rm bal}(K[mp^\infty],T)$ introduced in $\S\ref{subsection:Selmer}$.

Next we are interested in the non-triviality of our Euler system in terms of $L$-values. To this end, in $\S\ref{sec:Lp}$ we use some basic instances of Langlands functoriality to deduce from the work of Eischen--Harris--Li--Skinner \cite{EHLS} the construction of a $p$-adic $L$-function
\[
L_p({\rm ad}^0(g_K)\otimes\psi)\in{\rm Frac}\,\Lambda^{\rm ac}
\]
interpolating the central $L$-value $L(V,0)$ and its twists by a $p$-adic family of anticyclotomic Hecke characters. Here $\Lambda^{\rm ac}$ is the Iwasawa algebra of the Galois group $\Gamma^{\ac}$ of the anticyclotomic $\Z_p$-extension $K_\infty/K$. Denoting by $\kapinftys$ the image of $\kapinftyones$ in ${\rm Sel}_{\rm unb}(K_\infty,T)$, we can then prove the following.

Write $(p)=\fp\overline{\fp}$, with $\fp$ the prime of $K$ above $p$ induced by $\iota_p$. For the following result, let $K_{\infty,\overline{\fp}}$ be the $\mathbb Z_p$-extension of $K$ unramified outside $\overline{\fp}$, and let $\mathscr{F}_{\overline{\fp}}^{\rm bal}(\Trep)$ denote the subspace of $\Trep$ defined in Section \ref{subsection:Selmer}.

%Building on the explicit reciprocity law of \cite{BSV},

\begin{thmB}[Corollary~\ref{cor:ERL}]\label{thm:intro-B}
Under some technical hypotheses on $\psi$, there is a Perrin-Riou big logarithm map $\mathfrak{Log}:H^1_{\rm Iw}(K_{\infty,\overline{\fp}},\mathscr{F}_{\overline{\fp}}^{\rm bal}(\Trep))\longrightarrow\Z_p^{\rm ur}\dBr{\Gamma^{\rm ac}}$ such that
%sending ${\rm res}_{\overline{\fp}}(\kapinfty)$ to a square-root of the product
\[
\mathfrak{Log}({\rm res}_{\overline{\fp}}(\kapinftys))^2=L_p({\rm ad}^0(g_K)\otimes\psi)\cdot\mathscr{L}_{\fp}^{\rm Katz}(\psi)^{-,\iota}
\]
up to multiplication by an element in $\overline{\Q}_p^\times$,  where $\mathscr{L}_{\fp}^{\rm Katz}(\psi)^{-,\iota}$ is an anticyclotomic projection of Katz's $p$-adic $L$-function.
\end{thmB}

The proof of this result builds on the explicit reciprocity law of \cite{BSV} and a factorization formula for Hsieh's triple product $p$-adic $L$-function (see Theorem~\ref{thm:factor}). This factorization is a $p$-adic manifestation of the Artin formalism arising from the decomposition
\begin{equation}\label{eq:decV-intro}
%V_g\otimes V_{g^*}(\psi^{-1})(1-c)
\Vrep\simeq V\oplus V',
\end{equation}
where $V'=\Q_p(\psi^{-1})(1-k/2)$, and may be seen as an anticyclotomic analogue of Dasgupta's factorization \cite{Das}. However, the proof in our case is largely simplified by the fact that the $p$-adic $L$-functions involved have overlapping ranges of $p$-adic interpolation.
%Indeed, in our case the weight ranges for which $s=0$ is a central critical value of the $L$-function attached to $V$ and $V'$ is given by the following table:
%\begin{center}
%\begin{tabular}{c|c|c|}
%\cline{2-3}
% & $V$ & $V'$ \\
% \hline
%\multicolumn{1}{|c|}{$k\geq 2l$} & critical & critical  \\
% \hline
% \multicolumn{1}{|c|}{$k< 2l$} & noncritical & critical  \\
% \hline
%
%\end{tabular}
%\end{center}
%while the $p$-adic $L$-function for $\Vrep$ interpolates the central %critical values $L(\Vrep,0)$ in the range $k\geq 2l$.

The technical hypotheses on $\psi$ %in Theorem~\ref{thm:intro-B}
are used to ensure that the congruence ideal of a Hida family attached to $\psi$ is generated by a second anticyclotomic projection of Katz's $p$-adic $L$-function, which in turn interpolates the ratio between two different types of periods.

\subsection{Applications}

Using Kolyvagin's methods, as developed by Jetchev--Nekov\'{a}\v{r}--Skinner \cite{JNS} in the split anticyclotomic setting, we can deduce bounds on Selmer groups from the non-triviality of our Euler system.  Our main result in this direction is the proof of new cases of the Bloch--Kato conjecture \cite{BK} in rank zero.

For the statement, we denote by $\varepsilon_\ell$ the epsilon factor attached to the Weil--Deligne representation associated with the restriction of $\Ind_{K}^{\bb{Q}}(\Vrep)$ to $G_{\Q_\ell}$. It is then known that the sign $\varepsilon(\Vrep)$ in the functional equation for $L(\Vrep,s)$ is given by
\[
\varepsilon(\Vrep)=\prod_{\ell\leq\infty}\varepsilon_\ell,
\]
where $\varepsilon_\infty=+1$ if $k\geq 2l$ and $-1$ if $2\leq k<2l$. On the other hand, here we say that $V$ has ``big image'' if it safisfies the explicit conditions in Proposition~\ref{prop:existence-of-sigma}.

\begin{thmC}[Theorem~\ref{thm:thmC}]
In addition to the above hypotheses, assume that:
%Assume hypotheses (h1)--(h6), and in addition that:
\begin{itemize}
\item[\rm (a)] $\varepsilon_{\ell} = +1$ for all primes $\ell\mid N_gN_\psi$,
\item[\rm (b)] ${\rm gcd}(N_g,N_\psi)$ is squarefree,
\item[\rm (c)] $g$ is non-Eisenstein mod $p$,
\item[\rm (d)] $V$ has big image,
\item[\rm (e)] $L(\theta_\psi,k/2) \neq 0$.
\end{itemize}
If $k\geq 2l$ then the following implication holds:
\[
L(V,0)\neq 0\quad\Longrightarrow\quad \Sel(K,V)=0,
\]
where $\Sel(K,V)$ is the Bloch--Kato Selmer group.
\end{thmC}

Note that the hypotheses in Theorem~C imply that $L(V,s)$ has sign $+1$ in its functional equation, and so the nonvanishing of $L(V,0)$ is expected to hold generically.

We can also deduce applications to the Iwasawa main conjecture for $V$. More precisely, under certain hypotheses, Greenberg's general formulation of the Iwasawa main conjecture for motives \cite{Gr94} leads to the prediction that the unbalanced Selmer group ${\rm Sel}_{\rm unb}(K_\infty,A)$ defined in $\S\ref{sec:Selmer}$, where $A=V/T$, is $\Lambda^{\rm ac}$-cotorsion, with characteristic ideal generated by $L_p({\rm ad}^0(g_K)\otimes\psi)$. In the direction of this conjecture we can prove the following, where we let $\Z_p^{\rm ur}$ denote the completion of the ring of integers of the maximal unramified extension of $\Q_p$.

\begin{thmD}[Theorem~\ref{thm:thmD}]
In addition to the above hypotheses, assume that:
%Assume hypotheses (h1)--(h6), and in addition that:
\begin{itemize}
\item[\rm (a)] $\varepsilon_{\ell} = +1$ for all primes $\ell\mid N_gN_\psi$,
\item[\rm (b)] ${\rm gcd}(N_g,N_\psi)$ is squarefree,
\item[\rm (c)] $g$ is non-Eisenstein mod $p$,
\item[\rm (d)] $V$ has big image,
\item[\rm (e)] $\theta_\psi$ has global root number $\varepsilon(\theta_\psi)=+1$.
\end{itemize}
If the $p$-adic $L$-function $L_p({\rm ad}^0(g_K)\otimes\psi)$ is nonzero, then the Pontryagin dual of ${\rm Sel}_{\rm unb}(K_\infty,A)$ is $\Lambda^{\rm ac}$-torsion, with
\[
\Char_{\Lambda^{\rm ac}}\bigl({\rm Sel}_{\rm unb}(K_\infty,A)^\vee\bigr)\supset\bigl(L_p({\rm ad}^0(g_K)\otimes\psi)\cdot\mathscr{L}_{\fp}^{\rm Katz}(\psi)^{-,\iota}\bigr)
\]
in $\Z_p^{\rm ur}\hat\otimes_{\Z_p}\Lambda^{\rm ac}[1/p]$.
%\end{enumerate}
\end{thmD}

%\begin{remark}
Note that the presence of $\mathscr{L}_{\fp}^{\rm Katz}(\psi)^{-,\iota}$ in the divisibility of Theorem~D is analogous to the appearance of the Kubota--Leopoldt $p$-adic $L$-function in the divisibility towards the Iwasawa main conjecture for the Galois representation attached to the symmetric square of a modular form in \cite[Thm.~B]{LZ}.

In fact, the present work originated from an attempt to develop anticyclotomic analogues of the results in [\emph{op.\,cit.}]. In particular, the idea of modifying the diagonal Euler system classes of \cite{ACR} to obtain the correct norm relations (see $\S\ref{subsec:modified}$) was adopted from their work.

%\subsection{Relation to prior work}

%\begin{remark}
%As it will be clear to the reader, this paper owes to the work of Loeffler--Zerbes \cite{LZ} on the Iwasawa theory for the symmetric square of a modular form. In fact, the present work originated from an attempt to develop anticyclotomic analogues of the results contained in \emph{op.\,cit.}. In particular, the idea of modifying the diagonal cycle classes of \cite{ACR} to obtain the correct norm relations (see $\S\ref{subsec:modified}$) was adopted from their work.
%\end{remark}

\subsection{Outline of the paper}

We begin by introducing in $\S\ref{sec:Selmer}$ our set-up and Galois representation of interest, and various Selmer groups associated with it. In $\S\ref{section:bottomclass}$ we describe in detail the construction of the diagonal cycle class giving rise to the bottom class of our Euler system, and study its behaviour according to a certain sign (given by $\varepsilon(\theta_\psi)$ in the notations of Theorem~D). The results of this section,
which are developed in a slightly more general setting than the rest of the paper, are unnecessary for the proof of our main results, but they are included here for completeness (in particular,  Proposition~\ref{prop:sign-behaviour} might be of independent interest).
%; also, they illustrate some of the special features of the anticyclotomic setting.
In $\S\ref{sec:Lp}$ we introduce the different $p$-adic $L$-functions that appear in our picture, including an analogue of the Hida--Schmidt $p$-adic $L$-function deduced from the work of Eischen\,\emph{et.\,al.} on $p$-adic $L$-functions for unitary groups, and prove the aforementioned analogue of Dasgupta's factorization. Finally, in $\S\ref{sec:ES}$ we give the construction of our Euler system by suitably modifying the diagonal cycle Euler system classes constructed in our previous work \cite{ACR}, and in $\S\ref{sec:verify}$ and $\S\ref{sec:applications}$ we apply this to deduce the arithmetic applications highlighted in the Introduction.

%We begin by introducing the general set-up about Galois representations and Selmer groups, which is used to construct the bottom class of an Euler system. Then, we move to the discussion of $p$-adic $L$-functions, which is based on different works around unitary groups available in the literature. We continue by constructing an Euler system for the symmetric square of a modular form, by defining a suitable modification of our previous constructions. Finally, we discuss the applications to the Bloch--Kato conjecutre and the Iwasawa main conjecture we have underlined in the Introduction.

\subsection{Acknowledgements}
It is a pleasure to thank David Loeffler and  Chris Skinner for their very valuable advice in connection with this work. We are also grateful to Ellen Eischen and Xin Wan for correspondence regarding the subject of this note, and to Shilin Lai and Sam Mundy for several helpful conversations. Finally, we thank the anonymous referees for a careful reading of the text, whose comments notably contributed to improve the exposition of the article.

During the preparation of this paper, F.C. was partially supported by the NSF grant DMS-1946136 and DMS-2101458; O.R. was supported by the Royal Society Newton International Fellowship NIF\textbackslash R1\textbackslash 202208.

\section{Galois representations and Selmer groups}\label{sec:Selmer}

In this section we introduce our Galois representations of interest and the Selmer groups associated with them that we shall be studying.

\subsection{Galois representations}\label{subsec:Galrep}

Let $g=\sum_{n=1}^\infty a_n(g)q^n\in S_l(N_g,\chi_g)$ be an ordinary newform of weight $l\geq 2$, level $N_g$, and nebentypus $\chi_g$. Let $p>2$ be a prime and let $E=L_\mathfrak{P}$ be a finite extension of $\bb{Q}_p$ with ring of integers $\mathcal{O}$ arising as the completion of the Hecke field $L$ of $g$ at a prime $\mathfrak{P}$ of $L$ above $p$. By work of Eichler--Shimura and Deligne, there is a two-dimensional representation
$$
\rho_g\,:\,G_\bb{Q}\longrightarrow \GL_E(V_g)\simeq\GL_2(E)
$$
unramified outside $pN_g$ and characterized by the property
$$
{\rm trace}\,\rho_g(\Fr_{q})=a_{q}(g)
$$
for all primes $q\nmid pN_g$, where $\Fr_{q}$ denotes an arithmetic Frobenius element at $q$. %(Note that this is in fact the dual of the $p$-adic representation constructed by Deligne.)
Let $Y_1(N_g)$ be the open modular curve over $\bb{Q}$ parameterizing pairs $(A,P)$ consisting of an elliptic curve $A$ and a point $P\in A$ of order $N_g$. Let $\mathscr{L}_{l-2}$ is the sheaf introduced in \cite[\S2.3]{BSV}. As in \cite{ACR}, we shall work with the geometric realization of $V_g$ arising as the maximal quotient of
\[
H^1_{\et}(Y_{1}(N_g)_{\overline{\bb{Q}}},\mathscr{L}_{l-2}(1))\otimes_{\Z_p}E
\]
on which the dual Hecke operators $T_q'$ and $\langle d\rangle'$ act as multiplication by $a_q(g)$ and $\chi_g(d)$ for all primes $q\nmid N_g$ and all $d\in(\Z/N_g\Z)^\times$. We also let $T_g\subset V_g$ be the $\mathcal{O}$-lattice defined by the natural image of
\begin{align*}
H^1_{\et}(Y_{1}(N_g)_{\overline{\bb{Q}}}, \mathscr{L}_{l-2}(1))\otimes_{\Z_p}\cl{O}
\end{align*}
under the quotient map $H^1_{\et}(Y_{1}(N_g)_{\overline{\bb{Q}}},\mathscr{L}_{l-2}(1))\otimes_{\Z_p}E\twoheadrightarrow V_g$.

Throughout the following, we shall assume that $g$ is not of CM-type.

\subsection{The adjoint representation}\label{subsec:adjoint}

%Let $V = \Ad^0 g$ stand for the Galois representations attached to the symmetric square. As it is recalled in \cite[\S3.2]{LZ}, it is endowed with a 3-step filtration \[ V = \mathscr{F}^0 V \supset \mathscr{F}^1 V \supset \mathscr{F}^2 V \supset \mathscr{F}^3 V = 0. \]

%This induces an obvious three-step filtration
%\[
%0\subset\mathscr{F}^3\mathbb V_{f \Ad^0 g}^{\dag} \subset\mathscr{F}^2\mathbb V_{f \Ad^0 g}^\dag\subset\mathscr{F}^1\mathbb V_{f \Ad^0 g}^\dag \subset \mathbb V_{f \Ad^0 g}^\dag
%\]
%by $G_{\mathbb Q_p}$-stable submodules of ranks 1, 3, and 5, respectively, given by
%\begin{equation}
%\begin{split}\label{eq:+}
%\mathscr{F}^1\mathbb V_{f \Ad^0 g}^\dag &= (\mathbb V_{f} \hat\otimes_{\mathcal{O}} \mathscr F^1 V + \mathbb V_f^+ \hat\otimes_{\mathcal {O}} \mathscr F^0 V)(\Xi_{fgg}),\\
%\mathscr{F}^2\mathbb V_{f \Ad^0 g}^\dag &= (\mathbb V_f \hat\otimes_{\mathcal{O}} \mathscr F^2 V + \mathbb V_f^+ \hat\otimes_{\mathcal {O}} \mathscr F^1 V)(\Xi_{fgg}),\\
%\mathscr{F}^3\mathbb V_{f \Ad^0 g}^{\dag}&=\mathbb V_{f}^+\hat\otimes_{\mathcal{O}} \mathscr F^2 V (\Xi_{fgg}).
%\end{split}
%\end{equation}

%Here, $\Xi_{fgg}$ is the twist introduced in \cite{ACR}. Note that this discussion also applies to the setting where $f$ and $g$ are allowed to vary over a Hida family. We now suppose that $f$ is a CM form.

Let $K$ be an imaginary quadratic field of discriminant $-D_K<0$. Let $\psi$ be a Hecke character of $K$ of infinity type $(1-k,0)$ for some even integer $k\geq 2$ and central character equal to $\varepsilon_K$, the quadratic character attached to $K/\Q$ (thus the associated theta series $\theta_\psi$ has trivial nebentypus). We assume that $\psi$ has conductor $\mathfrak{c}\subset\cO_K$ prime to $p$ and, upon enlarging $\cO$ if necessary, that its $p$-adic avatar $\psi_{\frk{P}}$ takes values in $\cO$.

%Let $\mathfrak{c}$ be an ideal of $\cO_K$ coprime to $p$, and fix a Hecke character $\psi_0$ of infinity type $(-1,0)$,  conductor $\mathfrak{c}$ or $\frk{cp}$, and central character $\varepsilon_K\omega^{2n}$ for some interger $n$ (i.e. the conductor is $\frk{c}$ if $p-1\mid r_1$ and $\frk{cp}$ otherwise). Let $\mathcal{O}$ be the ring of integers of a finite extension of $\Q_p$ containing the values of $\psi_0$.

%Let $\lambda$ be the unique (since $p\nmid h_K$) Hecke character of infinity type $(-1,0)$ and conductor $\fp$ whose $p$-adic avatar factors through $\Gamma_\fp$. Let $\psi=\psi_0\lambda^{2n}\mathbf{N}^n$, which is a Hecke character of infinity type $(-1-n,n)$ and central character $\varepsilon_K$. Let $\psi_\frk{P}$ be its $p$-adic avatar.

%In this paper we are interested in the arithmetic of the following $p$-adic representation.

\begin{definition}\label{def:V}
Let $V$ be the $E$-valued $G_K$-representation given by
\[
V:=\ad^0(V_g)(\psi_{\frk{P}}^{-1})(1-k/2),
\]
where $\ad^0(V_g)\subset{\rm End}_E(V_g)$ denotes the adjoint representation on the trace-zero endomorphisms of $V_g$.
\end{definition}

Let $g^*=g\otimes\chi_g^{-1}$ be the twist of $g$ by the inverse of its nebentypus. We shall study the arithmetic of $V$ by exploiting the decomposition
\begin{equation}\label{eq:dec-V}
\Vrep:=V_g\otimes V_{g^\ast}(\psi_{\frk{P}}^{-1})(1-c)\simeq V\oplus V',
\end{equation}
where $c=(k+2l-2)/2$ and $V'=E(\psi_\mathfrak{P}^{-1})(1-k/2)$.

\subsection{Selmer groups}\label{subsection:Selmer}

From now on, we assume that $p$ is a prime of good ordinary reduction for $g$ such that
\begin{equation}\label{eq:split}
\textrm{$(p)=\fp\overline{\fp}$ splits in $K$,}
\end{equation}
with $\fp$ the prime of $K$ above $p$ determined by our fixed embedding $\iota_p:\overline{\Q}\hookrightarrow\overline{\Q}_p$.

%\begin{hypotheses}
%\begin{enumerate}
%\item $(p)=\fp\overline{\fp}$ splits in $K$,
%\item $p\nmid h_K$, the class number of $K$,
%\item the residual representation $\bar{\rho}_g$ is absolutely irreducible and $p$-distinguished.
%\end{enumerate}
%\end{hypotheses}

%\begin{remark}
%The residual representation $\bar{\rho}_{\theta_\psi}$ is absolutely irreducible and $p$-distinguished by the definition of $\psi$ (see \cite[Rmk.~5.1.3]{LLZ}).
%\end{remark}

%\begin{remark}
%Similarly as in \cite{ACR}, the hypothesis $p\nmid h_K$ is inessential and is introduced for simplicity.
%\end{remark}

By $p$-ordinarity, the Galois representation $V_g$ is equipped with a $G_{\Q_p}$-stable filtration
\[
0\longrightarrow V_g^+\longrightarrow V_g\longrightarrow V_g^-\longrightarrow 0
\]
with $V_g^{\pm}$ one-dimensional and the $G_{\Q_p}$-action on $V_g^-$ given by the unramified character sending an arithmetic Frobenius
${\rm Fr}_p$ to $\alpha_g$, the $p$-adic unit root of $x^2-a_p(g)x+\chi_g(p)p^{l-1}$. Of course, twisting these by $\chi_g^{-1}$ we obtain $V_{g^*}^\pm=V_g^\pm\otimes\chi_g^{-1}$.

Let $F/K$ be any finite extension and, for $v\mid p$ any prime of $F$ above $p$, define
%$\mathscr{F}_v^+(\Vrep)\subset\Vrep$ by
\begin{equation}\label{eq:local-bal}
\mathscr{F}_v^{\rm bal}(\Vrep):=\begin{cases}
(V_g^+\otimes V_{g^*}+V_g\otimes V_{g^*}^+)(\psi_{\frk{P}}^{-1})(1-c)&\textrm{if $v\mid\fp$,}\\[0.3em]
V_g^+\otimes V_{g^*}^+(\psi_{\frk{P}}^{-1})(1-c)&\textrm{if $v\mid\overline{\fp}$,}
\end{cases}
\end{equation}
and
\begin{equation}\label{eq:local-unb}
\mathscr{F}_v^{\rm unb}(\Vrep):=\begin{cases}
\Vrep&\textrm{if $v\mid\fp$,}\\[0.3em]
\{0\}&\textrm{if $v\mid\overline{\fp}$,}
\end{cases}
\end{equation}
and, for $\any\in\{{\rm bal},{\rm unb}\}$, put
$\mathscr{F}_v^{\any}(V)=\mathscr{F}_v^{\any}(\Vrep)\cap V$ and $\mathscr{F}_v^{\any}(V')=\mathscr{F}_v^{\any}(\Vrep)\cap V'$.

Fix $\Sigma$ any finite set of places of $K$ containing $\infty$ and the primes dividing $pN_gN_\psi$. With a slight abuse of notation, for any finite extension of $F/K$ we also denote by $\Sigma$ the set of places of $F$ lying over the places in $\Sigma$, and denote by $G_{F,\Sigma}$ the Galois group of the maximal extension of $F$ unramified outside $\Sigma$. Further, for any non-archimedean field $F_v$, we write $F_v^{\rm nr}$ for the maximal unramified extension of $F_v$.

\begin{definition}\label{def:Sel-ad}
Let $F/K$ be a finite extension, and for $M\in\{\Vrep,V,V'\}$ and $\any\in\{{\rm bal},{\rm unb}\}$ define the Selmer group ${\rm Sel}_\any(F,M)$ by
\[
\Sel_\any(F,M) = \ker \bigg(
H^1(G_{F,\Sigma},M) \longrightarrow
\prod_{v\mid p}\frac{H^1(F_v,M)}{H^1_{\any}(F_v,M)}\times\prod_{v \in \Sigma,v\nmid p\infty}H^1(F_v^{\rm nr},M)\bigg),
\]
where
\[
H_{\any}^1(F_v,M)={\rm im}(H^1(F_v,\mathscr{F}_v^\any(M))\longrightarrow H^1(F_v,M)).
\]
We call ${\rm Sel}_{\rm bal}(F,M)$ (resp. ${\rm Sel}_{\rm unb}(F,M)$) the \emph{balanced} (resp. \emph{unbalanced}) Selmer group.
\end{definition}

\begin{remark}\label{rem:shapiro}
Let $f=\theta_\psi$ be the weight $k$ eigenform associated with $\psi$, and denote by $V_{fgg^*}:=V_{f}\otimes V_g\otimes V_{g^*}(1-c)$ the Kummer self-dual twist of the Galois representation attached to $(f,g,g^*)$. Since $V_f = \Ind^{\mathbb Q}_K \psi$, one can easily check that the isomorphism given by Shapiro's lemma
\[
H^1(\Q,V_{fgg^*})\simeq H^1(K,\Vrep)
\]
identifies the Selmer groups ${\rm Sel}_{\rm bal}(\Q,V_{fgg^*})$ and ${\rm Sel}_{f}(\Q,V_{fgg^*})$ considered in \cite[Def.~7.5]{ACR} with the above ${\rm Sel}_{\rm bal}(K,\Vrep)$ and ${\rm Sel}_{\rm unb}(K,\Vrep)$, respectively.
\end{remark}

Put $\Trep=T_g\otimes T_{g^*}(\psi_\mathfrak{P}^{-1})(1-c)$. Then the decomposition (\ref{eq:dec-V}) induces a decomposition
\[
\Trep\simeq T\oplus T',
\]
where $T$ and $T'$ are lattices in $V$ and $V'$, respectively. We also set
\[
\Arep=\Vrep/\Trep,\quad A=V/T,\quad A'=V'/T'.
\]
Then, for $\any\in\{{\rm bal},{\rm unb}\}$ and $M\in\{\Trep,T,T',\Arep,A,A'\}$, we define the local conditions $H^1_\any(F_v,M)$ from the local conditions above by propagation, and use them to define the Selmer groups ${\rm Sel}_\any(K,M)$ using the same recipe as in Definition~\ref{def:Sel-ad}. Finally, for $M_1\in\{\Trep,T,T'\}$ and $M_2\in\{\Arep,A,A'\}$, we put
\[
{\rm Sel}_?(K_\infty,M_1):=\varprojlim_n{\rm Sel}_?(K_n,M_1),\quad
{\rm Sel}_?(K_\infty,M_2):=\varinjlim_n{\rm Sel}_?(K_n,M_2),
\]
where the limits are with respect to corestriction and restriction, respectively.

To help orient the reader, we note the following simple relation between the different Selmer groups introduced above.

\begin{proposition}
The decomposition $\Vrep=V\oplus V'$ induces isomorphisms
\begin{align*}
{\rm Sel}_{\rm bal}(K_\infty,\Trep)&\simeq{\rm Sel}_{\rm bal}(K_\infty,T)\oplus{\rm Sel}(K_\infty,T'),\\
{\rm Sel}_{\rm unb}(K_\infty,\Trep)&\simeq{\rm Sel}_{\rm unb}(K_\infty,T)\oplus{\rm Sel}(K_\infty,T'),
\end{align*}
where ${\rm Sel}(K_\infty,T')$ is the Bloch--Kato Selmer group for $T'$.
\end{proposition}

\begin{proof}
It suffices to show that for any finite extension $F/K$ we have
\[
{\rm Sel}_{\rm bal}(F,V')\simeq{\rm Sel}_{\rm unb}(F,V')\simeq{\rm Sel}(F,V'),
\]
where ${\rm Sel}(F,V')$ is the Bloch--Kato Selmer group of $V'=E(\psi_{\mathfrak{P}}^{-1})$, which is given by
\[
{\rm Sel}(F,V')=\ker\biggl(H^1(G_{F,\Sigma},V')\rightarrow \prod_{v\vert\bar{\fp}}H^1(F_{v},V')\times\prod_{v\in\Sigma, v\nmid p}H^1(F_v^{\rm nr},V')\biggr)
\]
(see \cite[\S{1.1}]{AH-ord} or \cite[\S{1.2}]{arnold}). For ${\rm Sel}_{\rm unb}(F,V')$ this is clear from (\ref{eq:local-unb}); for ${\rm Sel}_{\rm bal}(F,V')$ it follows by noting that the subspace  $\mathscr{F}_v^{\rm bal}(\Vrep)\subset\Vrep$ in (\ref{eq:local-bal}) contains $V'$ for $v\mid\fp$ and intersects trivially with it for $v\mid\bar{\fp}$.
\end{proof}

%\end{comment}

\section{Construction of the bottom class}
\label{section:bottomclass}

In this section, we recall the construction of a $\Lambda$-adic cohomology class associated with the triple product of three modular forms as explained in \cite{BSV}. We follow the exposition in [\emph{op.\,cit.}] with slight modifications and specializing the discussion to the case of interest in this paper. At the end of this section we analyze the behaviour of this cohomology class depending on the sign of one of the modular forms.

This section is independent of the rest of the paper, and the reader solely interested in the results stated in the Introduction can proceed to Section~\ref{sec:Lp}.

%Now we turn our attention to the bottom class of the system. To that end, let us briefly recall the construction of the class $\kapinfty:=\kapinftyone$. The construction of this class follows essentially from \cite[\S8.1]{BSV}. We summarize it here for convenience, with slight modifications.

Let $f$ and $g$ be newforms of weight $k=r_1+2$ and $l=r_2+2$, level $N_f$ and $N_g$ and character $\chi_f=1$  and  $\chi_g$, respectively. We assume that $p\nmid 2 N_f N_g$ and that both $f$ and $g$ are ordinary at $p$. We denote by $h=g^\ast$ the newform obtained by conjugating the Fourier coefficients of $g$. Let $L$ be a finite extension of $\bb{Q}$ containing the Fourier coefficients of $f$ and $g$ and let $E=L_\mathfrak{P}$ be its completion at a prime $\mathfrak{P}$ above $p$, with ring of integers $\mathcal{O}$. Define $N=\lcm(N_f,N_g)$.

Consider the Iwasawa algebra $\Lambda=\bb{Z}_p\dBr{1+p\bb{Z}_p}$. There exist finite flat $\Lambda$-modules $\Lambda_\hf$ and $\Lambda_\hg$  and primitive Hida families $\hf\in \Lambda_\hf\dBr{q}$ and $\hg\in\Lambda_\hg\dBr{q}$ passing through the ordinary $p$-stabilizations $f_\alpha$ and $g_\alpha$ of $f$ and $g$, respectively. Let $\hh=\hg^\ast$ be the Hida family $\hg\otimes\chi_g^{-1}$, which passes through the ordinary $p$-stabilization $g_\alpha^\ast$ of $g^\ast$. Our conventions for Hida families are those described in \cite[\S5.1]{ACR}.

Let $\mu_{p-1}$ denote the group of roots of unity in $\bb{Z}_p^\times$ and consider the decomposition $\mathbb Z_p^{\times} \cong \mu_{p-1}\times (1+p \mathbb Z_p)$. Let $\omega:\bb{Z}_p^\times\rightarrow \mu_{p-1}\subset \bb{Z}_p^\times$ be the map defined by projection onto the first factor, according to the previous decomposition. Also, for an element $z\in \bb{Z}_p^\times$, we denote by $\langle z \rangle$ its projection onto the second factor (alternatively, $\langle z \rangle = z/\omega(z)$).

Let $\text{Cont}(\bb{Z}_p,\Lambda)$ be the $\Lambda$-module of continuous functions on $\bb{Z}_p$ with values in $\Lambda$. To make notation less cumbersome, we denote by $[z]$ the group-like element $[\langle z\rangle]$ in $\Lambda$. For each integer $i$, let $\kappa_i:\bb{Z}_p^\times\rightarrow \Lambda^\times$ be the character defined by $z\mapsto \omega^i(z)[z]$. We also define the sets $\mathsf{T}=\bb{Z}_p^\times\times\bb{Z}_p$ and $\mathsf{T}'=p\bb{Z}_p\times\bb{Z}_p^\times$. Then, we can define the $\Lambda$-modules
\begin{gather}
\cl{A}_{i}=\left\{ f:\mathsf{T}\rightarrow \Lambda\; \vert\; f(1,z)\in \text{Cont}(\bb{Z}_p,\Lambda) \text{ and } f(a\cdot t)=\kappa_i(a)\cdot f(t) \text{ for all } a\in\bb{Z}_p^\times,\, t\in \mathsf{T}\right\},\nonumber\\
\cl{A}_{i}'=\big\{ f:\mathsf{T}'\rightarrow \Lambda \; \vert\; f(pz,1)\in \text{Cont}(\bb{Z}_p,\Lambda) \text{ and } f(a\cdot t)=\kappa_i(a)\cdot f(t)
\text{ for all } a\in\bb{Z}_p^\times,\, \gamma\in \mathsf{T}'\big\},\nonumber\\
\cl{D}_{i}=\Hom_{\text{cont},\Lambda}(\cl{A}_{i},\Lambda),\quad
\cl{D}_{i}'=\Hom_{\text{cont},\Lambda}(\cl{A}_{i}',\Lambda).\nonumber
\end{gather}

We define in addition characters $\kappa_{f}^\ast, \kappa_{g}^\ast, \kappa_{g^*}^\ast, \kappa^\ast:\bb{Z}_p^\times\rightarrow \Lambda\hat{\otimes}\Lambda\hat{\otimes}\Lambda$ by
\begin{align*}
\kappa_{f}^\ast(z)&=\omega^{r_2-r_1/2}(z)[z]^{-1/2}\otimes[z]^{1/2}\otimes[z]^{1/2} \\
\kappa_{g}^\ast(z)&=\omega^{r_1/2}(z)[z]^{1/2}\otimes[z]^{-1/2}\otimes[z]^{1/2} \\
\kappa_{h}^\ast(z)&=\omega^{r_1/2}(z)[z]^{1/2}\otimes[z]^{1/2}\otimes[z]^{-1/2} \\
\kappa^\ast(z)&=\omega^{r_1/2+r_2}(z)[z]^{1/2}\otimes[z]^{1/2}\otimes[z]^{1/2}.
\end{align*}
We denote by $\bm{\kappa}^\ast$ the character of the Galois group $G_{\bb{Q}}$ defined by $\bm{\kappa}^\ast=\kappa^\ast\circ\epsilon_{\cyc}$, and similarly for the other characters introduced above.

Let $Y=Y_1(N,p)$ denote the same modular curve as in \cite[\S8.1]{BSV} and let $\Gamma=\Gamma_1(N,p)$ be the corresponding modular group. The function
\[
\mathbf{Det}:\mathsf{T}'\times \mathsf{T}\times \mathsf{T}\longrightarrow \Lambda\hat{\otimes}\Lambda\hat{\otimes}\Lambda,
\]
defined as in [\emph{loc.\,cit.}], yields an element in the group
\[
H^0_{\et}(Y,\bm{\cl{A}}_{r_1}'\otimes \bm{\cl{A}}_{r_2}\otimes\bm{\cl{A}}_{r_2}(-\kappa^\ast)).
\]
Here $\bm{\cl{A}}_{r_1}'$ and $\bm{\cl{A}}_{r_2}$ denote the \'etale sheaves associated with $\cl{A}_{r_1}'$ and $\cl{A}_{r_2}$, respectively, as explained in \cite[\S4.2]{BSV}. Then, with essentially the same notations as in [\emph{op.\,cit.}, \S8.1], we define the class
\[
\kappa^{(1)}=\frac{1}{a_p(\hf)}\mathtt{s}_{\hf\hg\hh}\circ(e_{\text{ord}}\otimes e_{\text{ord}}\otimes e_{\text{ord}})\circ(w_p\otimes 1\otimes 1)\circ\mathtt{K}\circ\mathtt{HS}\circ d_\ast(\mathbf{Det})
\]
inside the group
\[
H^1\bigl(\bb{Q},H^1(\Gamma,\cl{D}_{r_1}')^{\text{ord}}\hat{\otimes} H^1(\Gamma,\cl{D}_{r_2}')^{\text{ord}}\hat{\otimes} H^1(\Gamma,\cl{D}_{r_2}')^{\text{ord}}(2-\bm{\kappa}^\ast)\bigr).
\]

For a $\bb{Z}_p$-algebra $A$, let $L_{r_2}(A)$ be defined as in [\emph{op.\,cit.}, p.~17]. We will sometimes denote $L_{r_2}(\bb{Z}_p)$ simply by $L_{r_2}$. Let $\kappa_f^{1/2}:\bb{Z}_p^\times\rightarrow \Lambda_\hf^\times$ denote the map $z\mapsto \omega^{r_1/2}(z)[z]^{1/2}$ and let $\bm{\kappa}_f^{1/2}=\kappa_f^{1/2}\circ \epsilon_{\cyc}$. According to [\emph{op.\,cit.}, eq.~(90)], the $\Lambda$-module $H^1(\Gamma,\cl{D}_{r_2}')^{\text{ord}}$ specializes to $H^1(\Gamma,L_{r_2})^{\text{ord}}$ at weight $l=r_2+2$. Therefore, the class $\kappa^{(1)}$ yields a class
\[
\kappa^{(2)}\in H^1\bigl(\bb{Q},H^1(\Gamma,\cl{D}_{r_1}')^{\text{ord}}\otimes H^1(\Gamma,L_{r_2})^{\text{ord}}\otimes H^1(\Gamma,L_{r_2})^{\text{ord}}(2-r_2-\bm{\kappa}_f^{1/2})\bigr).
\]

We define Hecke operators $T_q'$, $[d]'_N$ acting on group cohomology as in \cite[\S5.3]{ACR}. Let $\bb{V}_\hf(N)$ be the maximal quotient of $H^1(\Gamma,\cl{D}_{r_1}')^{\text{ord}}(1)\otimes_{\Lambda}\Lambda_\hf$ on which the Hecke operators $T_q'$ for primes $q\nmid N$ act as multiplication by $a_q(\hf)$ and the diamond operators $[d]_N'$ act as multiplication by $\chi_f(d)$ (actually, the character $\chi_f$ is trivial in our case). We define $T_g(N)$ and $T_{g^\ast}(N)$ in a similar way as quotients of $H^1(\Gamma,L_{r_2}(\cl{O}))^{\text{ord}}(1)$. Also, let $\bb{V}_\hf$ be the maximal quotient of $H^1(\Gamma_1(N_f,p),\cl{D}_{r_1}')^{\text{ord}}(1)\otimes_{\Lambda}\Lambda_\hf$ on which the Hecke operators $T_q'$ act as multiplication by $a_q(\hf)$ and the diamond operators $[d]'_N$ act as multiplication by $\chi_f(d)$ and define $T_g$ and $T_{g^\ast}$ in a similar way as quotients of $H^1(\Gamma_1(N_g,p),L_{r_2}(\cl{O}))^{\text{ord}}(1)$.
%Since, under our assumptions, the $p$-adic Galois representations attached to $f$ and $g$ are residually irreducible, the modules $\bb{V}_\hf$, $\bb{V}_\hg$ and $\bb{V}_{\hg^\ast}$ are free of rank two over $\Lambda_\hf$, $\Lambda_{\hg}$ and $\Lambda_{\hg}$, respectively.

To shorten notation, we define
\begin{align*}
\bb{V}(\hf,g,g^\ast)&=\bb{V}_\hf\otimes T_g\otimes T_{g^\ast}(-1-r_2-\bm{\kappa}_f^{1/2}),\\
\bb{V}(\hf,g,g^\ast)(N)&=\bb{V}_\hf(N)\otimes T_g(N)\otimes T_{g^\ast}(N)(-1-r_2-\bm{\kappa}_f^{1/2}),\\
\bb{V}(\hf,g,g^\ast)_f&=\bb{V}_\hf^-\otimes T_g^+\otimes T_{g^\ast}^+(-1-r_2-\bm{\kappa}_f^{1/2}),\\
\bb{V}(\hf,g,g^\ast)_f(N)&=\bb{V}_\hf^-(N)\otimes T_g^+(N)\otimes T_{g^\ast}^+(N)(-1-r_2-\bm{\kappa}_f^{1/2}).
\end{align*}
We also introduce
\begin{align*}
M(\hf,g,g^\ast)_f&=\bb{V}_\hf^-\hat{\otimes} T_g^+\hat{\otimes} T_{g^\ast}^+(-2-2r_2)\hat{\otimes}\Lambda(\bm{\kappa}^{-1})[1/p],\\
M(\hf,g,g^\ast)_f(N)&=\bb{V}_\hf^-(N)\hat{\otimes} T_g^+(N)\hat{\otimes} T_{g^\ast}^+(N)(-2-2r_2)\hat{\otimes}\Lambda(\bm{\kappa}^{-1})[1/p],
\end{align*}
where $\bm{\kappa}:G_\bb{Q}\rightarrow \Lambda^\times$ is defined by $\bm{\kappa}(\sigma)=\omega^{r_1/2-r_2-1}(\epsilon_{\cyc}(\sigma))[\epsilon_{\cyc}(\sigma)]$.

The class $\kappa^{(2)}$ yields a class
\[
\kappa^{(2)}(\hf,g,g^\ast)\in H^1\left(\bb{Q},\bb{V}(\hf,g,g^\ast)(N)\right).
\]
This is the class defined in \cite[eq.~155]{BSV} specialized to weight $l$ in the second and third factors. It follows from [\emph{op.\,cit.}, Cor.~8.2] that the restriction at $p$ of this class belongs to the group
\[
H^1_{\bal}\left(\bb{Q}_p,\bb{V}(\hf,g,g^\ast)(N)\right).
\]

Let $S_{\Lambda_\hf}^{\ord}(N,\omega^{r_1})$ denote the space of Hida families of tame level $N$, character $\omega^{r_1}$ and with coefficients in $\Lambda_\hf$.
%This space is defined as in \cite[\S5]{BSV}, except that we require the coefficients of the corresponding $q$-expansions to be elements of the integral module $\Lambda_\hf$.
We denote by $S_{\Lambda_\hf}^{\ord}(N,\omega^{r_1})[\hf]$ the subspace of $S_{\Lambda_\hf}^{\ord}(N,\omega^{r_1})$ on which the Hecke operators $U_p$ and $T_\ell$ for $\ell\nmid N$ act with the same eigenvalues as on $\hf$. Let $S_l(\Gamma_1(N,p),\chi_g)[g_\alpha]$ denote the space of modular forms of weight $l$, level $\Gamma_1(N,p)$ and nebentypus $\chi_g$ which are eigenforms for the Hecke operators $U_p$ and $T_\ell$ for $\ell\nmid N$ with the same eigenvalues as $g_\alpha$. We similarly define $S_l(\Gamma_1(N,p),\chi_g^{-1})[g_\alpha^\ast]$. Then, a choice of level-$N$ test vectors $\breve{\hf}$, $\breve{g}$ and $\breve{h}$ for $\hf$, $g_\alpha$ and $g_{\alpha}^\ast$, respectively, is a choice of elements
\[
\breve{\hf}\in S_{\Lambda}^{\ord}(N,\omega^{r_1})[\hf],\quad  \breve{g}\in S_l(\Gamma_1(N,p),\chi_g)[g_\alpha],\quad \breve{h}\in S_l(\Gamma_1(N,p),\chi_g^{-1})[g_\alpha^\ast],
\]
each of which can be written, in terms of their $q$-expansions, as
\begin{gather*}
\breve{\hf}(q)=\sum_{0<d\mid N/N_f} r_d^{\breve{\hf}}\cdot\hf(q^d), \\
\breve{g}(q)=\sum_{0<d\mid N/N_g} r_d^{\breve{g}}\cdot g_\alpha(q^d), \\
\breve{h}(q)=\sum_{0<d\mid N/N_g} r_d^{\breve{h}}\cdot g_\alpha^\ast(q^d),
\end{gather*}
with $r_d^{\breve{\hf}}\in\Lambda_\hf$ and $r_d^{\breve{g}}, r_d^{\breve{h}}\in \cl{O}$. Let
\[
\varpi_{\breve{\hf}}^\ast:S_{\Lambda_\hf}^{\ord}(N_f,\omega^{r_1})\longrightarrow S_{\Lambda_\hf}^{\ord}(N,\omega^{r_1})
\]
denote the map defined by
\[
\Phi(q)\mapsto \sum_{0<d\mid N/N_f} r_d^{\breve{\hf}}\cdot \Phi(q^d).
\]
Similarly, we define
\begin{gather*}
\varpi_{\breve{g}}^\ast:S_l(\Gamma_1(N_g,p),\chi_g)\longrightarrow S_l(\Gamma_1(N,p),\chi_g), \\
\varpi_{\breve{h}}^\ast:S_l(\Gamma_1(N_g,p),\chi_g^{-1})\longrightarrow S_l(\Gamma_1(N,p),\chi_g^{-1}).
\end{gather*}
Therefore, we can write $\breve{\hf}=\varpi_{\breve{\hf}}^\ast(\hf)$,  $\breve{g}=\varpi_{\breve{g}}^\ast(g_\alpha)$ and $\breve{h}=\varpi_{\breve{h}}^\ast(g^\ast_\alpha)$. At the same time, for each $d\mid N/N_f$, the map $v_d:Y_1(N,p)\rightarrow Y_1(N_f,p)$, corresponding to multiplication by $d$ on the complex upper half-plane under the standard complex uniformizations, yields a pushforward map
\[
v_{d\ast}: H^1(\Gamma_1(N,p),\cl{D}_{r_1}')\longrightarrow H^1(\Gamma_1(N_f,p),\cl{D}_{r_1}')
\]
which induces a map
\[
v_{d\ast}:\bb{V}_\hf(N)\longrightarrow \bb{V}_\hf.
\]
Let $\varpi_{\breve{\hf}\ast}=\sum_{0<d\mid N/N_f} r_d^{\breve{\hf}} v_{d\ast}$. Let $\eta_{\breve{\hf}}$ and $\eta_{\hf}$ denote the differentials attached to $\breve{\hf}$ and $\hf$, respectively, in \cite[eq.~(122)]{BSV}. Then, %$\eta_{\breve{\hf}}=\varpi_{\breve{\hf}}^\ast(\eta_{\hf})$ and
\[
\langle x,\eta_{\breve{\hf}}\rangle= \langle \varpi_{\breve{\hf}\ast}(x),\eta_\hf\rangle \quad \text{for all }x\in \bb{V}_\hf^-(N).
\]
%The Eichler--Shimura isomorphism in \cite[eq.~(117)]{BSV}, following from the work of Ohta \cite{Ohta00} and King--Loeffler--Zerbes \cite{KLZ}, gives rise to pairings
%\begin{gather*}
%\langle\,,\,\rangle_{N,\hf}:\bb{V}_\hf(N)\times S_{\Lambda_\hf}^{\ord}(N,\omega^{r_1})[\hf] \longrightarrow \Lambda_\hf, \\
%\langle\,,\,\rangle_{N_f,\hf}:\bb{V}_\hf\times S_{\Lambda_\hf}^{\ord}(N_f,\omega^{r_1})[\hf] \longrightarrow \Lambda_\hf.
%\end{gather*}
\begin{comment}
The maps $\varpi_{\hf}^\ast$ and $\varpi_{\hf\ast}$ satisfy the following ``adjointness'' property:
\[
\langle x,\varpi_{\breve{\hf}}^\ast(y)\rangle_{N,\hf}=\langle \varpi_{\breve{\hf}\ast}(x),y\rangle_{N_f,\hf}.
\]
\end{comment}
Similarly, we can define maps
\[
\varpi_{\breve{g}\ast}, \varpi_{\breve{h}\ast}: H^1(\Gamma_1(N,p),L_{r_2})\longrightarrow H^1(\Gamma(N_g,p),L_{r_2}),
\]
which induce maps
\[
\varpi_{\breve{g}\ast}: T_g(N)\longrightarrow T_g,\quad \varpi_{\breve{h}\ast}: T_{g^\ast}(N)\longrightarrow T_{g^\ast}.
\]
Let $\omega_{\breve{g}}$, $\omega_{g_\alpha}$, $\omega_{\breve{h}}$ and $\omega_{g_\alpha^\ast}$ denote the differentials attached to $\breve{g}$, $g_\alpha$, $\breve{h}$ and $g_\alpha^\ast$, respectively, in \cite[eq.~(30)]{BSV}. Then, %$\omega_{\breve{g}}=\varpi_{\breve{g}}^\ast(\omega_{g_\alpha})$, $\omega_{\breve{h}}=\varpi_{\breve{h}}^\ast(\omega_{g_\alpha^\ast})$, and the following identities hold:
\begin{gather*}
\langle x,\omega_{\breve{g}}\rangle= \langle \varpi_{\breve{g}\ast}(x),\omega_{g_\alpha}\rangle \quad \text{for all }x\in T_g^+(N), \\
\langle x,\omega_{\breve{h}}\rangle= \langle \varpi_{\breve{g}\ast}(x),\omega_{g_\alpha^\ast}\rangle \quad \text{for all }x\in T_{g^\ast}^+(N)
\end{gather*}
\begin{comment}
The pairs $(\varpi_{\breve{g}}^\ast,\varpi_{\breve{g}\ast})$ and $(\varpi_{\breve{h}}^\ast,\varpi_{\breve{h}\ast})$ satisfy an ``adjointness'' property similar to the one stated for $(\varpi_{\breve{\hf}}^\ast,\varpi_{\breve{\hf}\ast})$.
\end{comment}

For a choice of level-$N$ test vectors $\breve{\hf}=\varpi_{\breve{\hf}}^\ast(\hf)$,  $\breve{g}=\varpi_{\breve{g}}^\ast(g_\alpha)$, $\breve{h}=\varpi_{\breve{h}}^\ast(g^\ast_\alpha)$,
we have a map
\[
\frk{Log}(\breve{\hf},\breve{g},\breve{h}):H^1_{\bal}\left(\bb{Q}_p,\bb{V}(\hf,g,g^\ast)(N)\right)\longrightarrow \Lambda_\hf[1/p]
\]
obtained from the map defined in \cite[Prop.~7.3]{BSV} by specializing to weight $l$ the second and third variables. It follows from [\emph{op.\,cit.}, Thm.~A] that the image of $\res_p(\kappa^{(2)}(\hf,g,h))$ under the map above is an element $\mathscr{L}_p(\breve{\hf},\breve{g},\breve{h})\in \Lambda_\hf[1/p]$ such that, for all $k'\geq 2l$ satisfying $k'\equiv k\pmod{2(p-1)}$,
\[
\mathscr{L}_p(\breve{\hf},\breve{g},\breve{h})(k')=\frac{\langle \breve{\hf}_{k'}^w,\delta^t \breve{g}\times \breve{h}\rangle_{Np}}{\langle \breve{\hf}_{k'}^w,\breve{\hf}_{k'}^w\rangle_{Np}}.
\]

Let $\bar{\cl{L}}_{\breve{\hf}\breve{g}\breve{h}}$ be the map defined in \cite[Prop.~7.1]{BSV} specialized to weight $l$ in the second and third variables. % and let $\eta_{\breve{\hf}}=\varpi_{\breve{\hf}}^\ast(\eta_\hf)$, $\omega_{\breve{g}}=\varpi_{\breve{g}}^\ast(\omega_{g_\alpha})$, $\omega_{\breve{h}}=\varpi_{\breve{h}}^\ast(\omega_{g^\ast_\alpha})$ be the differentials introduced in [\emph{op.\,cit.}, eq.~(122)] and [\emph{op.\,cit.}, eq.~(30)].
Then, we obtain a map
\[
\langle\bar{\cl{L}}_{\breve{\hf}\breve{g}\breve{h}}(-),\eta_{\breve{\hf}}\omega_{\breve{g}}\omega_{\breve{h}}\rangle : H^1\left(\bb{Q}_p,M(\hf,g,g^\ast)_f(N)\right)\longrightarrow \Lambda_\hf\hat{\otimes}\Lambda[1/p].
\]
The map $\frk{Log}(\breve{\hf},\breve{g},\breve{h})$ is obtained by composing the natural projection
\[
H^1_{\bal}\left(\bb{Q}_p,\bb{V}(\hf,g,g^\ast)(N)\right)\longrightarrow H^1\left(\bb{Q}_p,\bb{V}(\hf,g,g^\ast)_f(N)\right)
\]
with a suitable specialization of the map above.

Now, from the previous discussion, we have that
\[
\langle \bar{\cl{L}}_{\breve{\hf}\breve{g}\breve{h}}(-),\eta_{\breve{\hf}}\omega_{\breve{g}}\omega_{\breve{h}}\rangle=\langle \bar{\cl{L}}_{\hf g_\alpha g^\ast_\alpha}((\varpi_{\breve{\hf},\ast}\otimes\varpi_{\breve{g},\ast}\otimes \varpi_{\breve{h},\ast})(-)),\eta_{\hf}\omega_{g_\alpha}\omega_{g_\alpha^\ast}\rangle,
\]
where the map $\bar{\cl{L}}_{\hf g_\alpha g^\ast_\alpha}$ is defined in a way analogous to the way in which the map $\bar{\cl{L}}_{\breve{\hf}\breve{g}\breve{h}}$ is defined in \cite[Prop.~7.1]{BSV}.
Therefore, as before, the composition of the natural projection
\[
H^1_{\bal}\left(\bb{Q}_p,\bb{V}(\hf,g,g^\ast)\right)\longrightarrow H^1\left(\bb{Q}_p,\bb{V}(\hf,g,g^\ast)_f\right)
\]
with a suitable specialization of the map
\[
\left\langle \bar{\cl{L}}_{\hf g_\alpha g^\ast_\alpha}(-),\eta_{\hf}\omega_{g_\alpha}\omega_{g^\ast_\alpha}\right\rangle : H^1\left(\bb{Q}_p,M(\hf,g,g^\ast)_f\right)\longrightarrow \Lambda_\hf\hat{\otimes}\Lambda[1/p]
\]
yields a map
\[
\frk{Log}(\hf,g_\alpha,g^\ast_\alpha):H^1_{\bal}\left(\bb{Q}_p,\bb{V}(\hf,g,g^\ast)\right)\longrightarrow \Lambda_\hf[1/p].
\]
Moreover, for any choice of test vectors $\breve{\hf}, \breve{g}, \breve{h}$ as above, we have
\[
\frk{Log}(\hf,g_\alpha,g^\ast_\alpha)\bigl(\res_p\bigl((\varpi_{\breve{\hf},\ast}\otimes\varpi_{\breve{g},\ast}\otimes \varpi_{\breve{h},\ast})(\kappa^{(2)}(\hf,g,g^\ast))\bigr)\bigr)=\mathscr{L}_p(\breve{\hf},\breve{g},\breve{h}).
\]

It follows from \cite[\S3.5-6]{Hs} and [\emph{op.\,cit.}, Thm.~7.1] that there exist level-$N$ test vectors $\breve{\hf}, \breve{g}, \breve{h}$ for which, under some technical assumptions, we have a precise formula for the specializations of $\mathscr{L}_p(\breve{\hf},\breve{g},\breve{h})$ at even weights $k'\geq 2l$. We fix such test vectors.
\begin{comment}
More precisely, we set
\[
\breve{\hf}=\omega^{4-k}(d_\hf)d_\hf^{-1}[\langle -Nd _\hf\rangle^{-1}]U_{d_\hf}w_N(\pi_1^\ast(\hf)),\quad \breve{g}=\pi_1^\ast (g), \quad \breve{h}=(N/N_g)^{-r_2}\pi_2^\ast(\hg^\ast),
\]
where $d_\hf=d_\hf^{\text{(II)}}$ is defined in [\emph{op.\,cit}, \S3.4], the operator $U_{d_\hf}$ is defined in [\emph{op.\,cit}, \S3.1] and the operator $w_N$ is defined in \cite[Lem.~6.1]{BSV}. We are taking into account the discrepancy between conventions noted in [\emph{op.\,cit}, Rmk.~6.2].
\end{comment}
\begin{comment}
This choice of test vectors determines degeneracy maps
\begin{align*}
\varpi_{\breve{\hf},\ast}:H^1(\Gamma,\cl{D}_{r_1}')^{\text{ord}}&\longrightarrow H^1(\Gamma_1(N_\psi,p),\cl{D}'_{r_1})^{\text{ord}},\\
\varpi_{\breve{g},\ast}:H^1(\Gamma,L_{r_2})^{\text{ord}}&\longrightarrow H^1(\Gamma_1(N_g,p),L_{r_2})^{\text{ord}},
\\
\varpi_{\breve{h},\ast}:H^1(\Gamma,L_{r_2})^{\text{ord}}&\longrightarrow H^1(\Gamma_1(N_g,p),L_{r_2})^{\text{ord}}.
\end{align*}
\end{comment}
Then, we define
\[
\kappa^{(3)}=(\varpi_{\breve{\hf},\ast}\otimes\varpi_{\breve{g},\ast}\otimes\varpi_{\breve{h},\ast})\kappa^{(2)}
\]
in the group
\[
H^1\bigl(\bb{Q},H^1(\Gamma_1(N_f,p),\cl{D}_{r_1}')^{\text{ord}}\hat{\otimes} H^1(\Gamma_1(N_g,p),L_{r_2})^{\text{ord}}\hat{\otimes} H^1(\Gamma_1(N_g,p),L_{r_2})^{\text{ord}}(2-r_2-\bm{\kappa}^{1/2}_f)\bigr)
\]
and let
\[
\kappa^{(3)}(\hf,g,g^\ast)\in H^1\bigl(\bb{Q},\bb{V}_\hf(-\bm{\kappa}_f^{1/2})\hat{\otimes}\ad(T_g)\bigr)
\]
be the class obtained from $\kappa^{(3)}$ by projection to the isotypic quotients for $\hf$, $g$ and $g^\ast$. Then
\[
\frk{Log}(\hf,g_\alpha,g^\ast_\alpha)\bigl(\res_p\bigl(\kappa^{(3)}(\hf,g,g^\ast)\bigr)\bigr)=\mathscr{L}_p(\breve{\hf},\breve{g},\breve{h}).
\]

Observe that the map $w_{N_g}:H^1(\Gamma_1(N_g,p),L_{r_2}(\cl{O}))^{\text{ord}}\rightarrow H^1(\Gamma_1(N_g,p),L_{r_2}(\cl{O}))^{\text{ord}}$ defined in \cite[\S4.1.2]{BSV} descends to a map
$w_{N_g}: T_g\rightarrow T_{g^\ast}$. Taking the Galois action into account, this is actually a map $T_g\rightarrow T_{g^\ast}(\chi_g)$. Similarly, we have a map $w_{N_g}: T_{g^\ast}\rightarrow T_g(\chi_g^{-1})$. (We are denoting all these maps in the same way in the hope that this will not cause any confusion.)

Let $s:T_g\otimes T_{g^\ast}\rightarrow T_{g^\ast}\otimes T_g$ be the map which interchanges the two factors. Then, the composition $\tilde{s}=(-N_g)^{-r_2}\, s\circ (w_{N_g},w_{N_g})$ defines an endomorphism of $\ad(T_g)=T_g\otimes T_{g^\ast}(-1-r_2)$. This endomorphism is in fact an involution.

\begin{lemma}
Consider the direct sum decomposition $\ad(T_g)=\ad^0(T_g)\oplus 1$. Then:
\begin{enumerate}
\item $\ad^0(T_g)$ is the $1$-eigenspace for $\tilde{s}$;
\item $1$ is the $-1$-eigenspace for $\tilde{s}$.
\end{enumerate}
\end{lemma}
\begin{proof}
As in \cite[p.~19]{BSV}, there is a bilinear form $L_{r_2}(\cl{O})\otimes L_{r_2}(\cl{O})\rightarrow \cl{O}\otimes\det^{-r_2}$. Via cup-product and the isomorphism $H^2_{\text{par}}(\Gamma_1(N_g,p),\cl{O})\simeq \cl{O}(1)$, we obtain a pairing
\[
H^1_{\text{par}}(\Gamma_1(N_g,p),L_{r_2}(\cl{O}))^{\text{ord}}\times H^1_{\text{par}}(\Gamma_1(N_g,p),L_{r_2}(\cl{O}))^{\text{ord}} \longrightarrow \cl{O}(r_2+1),
\]
where $H^1_{\text{par}}$ stands for parabolic cohomology as defined in \cite[p.~427]{GrSt}.
Since cup-product is anti-commutative in degree 1, the pairing above satisfies $\langle \alpha,\beta\rangle = (-1)^{r_2+1}\langle \beta,\alpha\rangle$ for any $\alpha,\beta$ in $H^1_{\text{par}}(\Gamma_1(N_g,p),L_{r_2}(\cl{O}))^{\text{ord}}$. On the other hand, the operator $w_{N_g}$ acting on $H^1_{\text{par}}(\Gamma_1(N_g,p),L_{r_2}(\cl{O}))^{\text{ord}}$ satisfies
$w_{N_g}^2=(-N_g)^{r_2}$ and $\langle w_{N_g}\alpha,w_{N_g}\beta\rangle=N_g^{r_2}\langle \alpha,\beta\rangle$ for any elements $\alpha,\beta\in H^1_{\text{par}}(\Gamma_1(N_g,p),L_{r_2}(\cl{O}))^{\text{ord}}$. Therefore
%for any elements $\alpha,\beta \in H^1_{\text{par}}(\Gamma_1(N_g,p),L_{r_2}(\cl{O}))^{\text{ord}}$,
we have
\[
\langle \alpha,w_{N_g}\beta\rangle=\frac{1}{N_g^{r_2}}\langle w_{N_g}\alpha,w_{N_g}^2\beta\rangle=(-1)^{r_2}\langle w_{N_g}\alpha,\beta\rangle=-\langle\beta,w_{N_g}\alpha\rangle.
\]
In particular, we deduce that $\langle\alpha,w_{N_g}\alpha\rangle=0$.

We can realize $T_g$ (resp. $T_{g^\ast}$) as the maximal quotient of $H^1_{\text{par}}(\Gamma_1(N_g,p),L_{r_2}(\cl{O}))^{\text{ord}}$ on which the Hecke operators $T_q'$ act as multiplication by $a_q(g)$ (resp. $a_q(g^\ast)$) and the diamond operators $[d]_{N_g}'$ act as multiplication by $\chi_g(d)$ (resp. $\chi_g(d)^{-1}$). Thus we obtain a commutative diagram
\[
\begin{tikzcd}
H^1_{\text{par}}(\Gamma_1(N_g,p),L_{r_2}(\cl{O}))^{\text{ord}}\times H^1_{\text{par}}(\Gamma_1(N_g,p),L_{r_2}(\cl{O}))^{\text{ord}} \arrow[r] \arrow[d] & \cl{O}(r_2+1) \arrow[d,equal]   \\
T_g\times T_{g^\ast}\arrow[r] & \cl{O}(r_2+1).
\end{tikzcd}
\]
Therefore, for any elements $\alpha,\beta\in T_g$, we have $\langle\alpha,w_{N_g}\beta\rangle=-\langle\beta,w_{N_g}\alpha\rangle$. The lemma follows easily from this.
\end{proof}

We will assume in the remaining of this section that $N_g\mid N_f$, so that $N=N_f$. Under this assumption, our test vectors are $\breve{\hf}=\hf$, $\breve{g}(q)=\pi_1^\ast(g_\alpha)$ and $\breve{h}=\pi_2^\ast(g_\alpha^\ast)$, up to multiplication by some constants in $\Frac{\Lambda_\hf}$ which do not affect the discussion that follows.

Let $s_{N_g}$ denote the operator which acts on the group
\[
H^1\bigl(\bb{Q},H^1(\Gamma,\cl{D}_{r_1}')^{\text{ord}}\hat{\otimes} H^1(\Gamma_1(N_g,p),L_{r_2})^{\text{ord}}\hat{\otimes} H^1(\Gamma_1(N_g,p),L_{r_2})^{\text{ord}}(2-r_2-\bm{\kappa}_f^{1/2})\bigr)
\]
by interchanging the second and third factors and define $\tilde{s}_{N_g}=(-N_g)^{-r_2} s_{N_g}\circ (1\otimes w_{N_g}\otimes w_{N_g})$.

\begin{propo}
\label{prop:sign-behaviour}
The class $\kappa^{(3)}(\hf,g,g^\ast)$ satisfies
\[
\tilde{s}_{N_g}\bigl(\kappa^{(3)}(\hf,g,g^\ast)\bigr)=-[N]^{-1/2}(w_{N}\otimes 1\otimes 1)\kappa^{(3)}(\hf,g,g^\ast).
\]
In particular, when we consider the direct sum decomposition
\[
H^1\bigl(\bb{Q},\bb{V}_\hf(-\bm{\kappa}_f^{1/2})\hat{\otimes}\ad(T_g)\bigr)=H^1\bigl(\bb{Q},\bb{V}_\hf(-\bm{\kappa}_f^{1/2})\hat{\otimes}\ad^0(T_g)\bigr)\oplus H^1\bigl(\bb{Q},\bb{V}_\hf(-\bm{\kappa}_f^{1/2})\bigr),
\]
the class $\kappa^{(3)}(\hf,g,g^\ast)$ lies in the summand
\begin{enumerate}
\item $H^1(\bb{Q},\bb{V}_\hf(-\bm{\kappa}_f^{1/2})\hat{\otimes}\ad^0(T_g))$, if $\varepsilon(f)=1$;
\item $H^1(\bb{Q},\bb{V}_\hf(-\bm{\kappa}_f^{1/2}))$, if $\varepsilon(f)=-1$.
\end{enumerate}
\end{propo}
\begin{proof}
We have the following commutative diagram
\[
\begin{tikzcd}
H^0_{\et}(Y,\bm{\cl{A}}_{r_1}'\otimes \bm{\cl{A}}_{r_2}\otimes\bm{\cl{A}}_{r_2}(-\kappa^\ast)) \arrow[r, "d_\ast"] \arrow[d, "w_N"] & H^4_{\et}(Y^3,\bm{\cl{A}}_{r_1}'\boxtimes \bm{\cl{A}}_{r_2}\boxtimes \bm{\cl{A}}_{r_2}(-\kappa^\ast)\otimes\bb{Z}_p(2)) \arrow[d, "{(w_N, w_N, w_N)}"] \\
H^0_{\et}(Y,\bm{\cl{A}}_{r_1}'\otimes \bm{\cl{A}}_{r_2}\otimes\bm{\cl{A}}_{r_2}(-\kappa^\ast)) \arrow[r, "d_\ast"] & H^4_{\et}(Y^3,\bm{\cl{A}}_{r_1}'\boxtimes \bm{\cl{A}}_{r_2}\boxtimes \bm{\cl{A}}_{r_2}(-\kappa^\ast)\otimes\bb{Z}_p(2)),
\end{tikzcd}
\]
where $w_N$ stands here for the operator defined in \cite[\S2.3.1]{BSV} and $(w_N,w_N,w_N)$ is defined in a similar way for the cohomology of $Y^3$. It follows from the definition of $\mathbf{Det}$ that $w_N(\mathbf{Det})=\mathbf{Det}$. Since $w_p w_N = [p]_N w_N w_p$ and $\mathtt{s}_{\hf\hg\hh}\circ([p]_N w_N\otimes w_N\otimes w_N)=[p]_N'(w_N\otimes w_N\otimes w_N)\circ \mathtt{s}_{\hf\hg\hh}$, it follows that
\[
(w_N\otimes w_N\otimes w_N)\kappa^{(1)}=\kappa_{\hf\hg\hh}^\ast(N)([p]_N\otimes 1\otimes 1)\kappa^{(1)}.
\]
Since $(w_N^2\otimes 1\otimes 1)$ acts as multiplication by $[-N]\otimes 1\otimes 1$, we deduce that
\[
(1\otimes w_N\otimes w_N)\kappa^{(1)}=\kappa_f^{\ast}(N)([p]'_N w_N\otimes 1\otimes 1)\kappa^{(1)}
\]
and therefore that
\[
N^{-r_2}(1\otimes w_N\otimes w_N)\kappa^{(2)}=\kappa_f^{-1/2}(N)([p]'_N w_N\otimes 1\otimes 1)\kappa^{(2)}.
\]
Let $s_N$ denote the operator which acts on the group
\[
H^1\bigl(\bb{Q},H^1(\Gamma,\cl{D}_{r_1}')^{\text{ord}}\hat{\otimes} H^1(\Gamma,L_{r_2})^{\text{ord}}\hat{\otimes} H^1(\Gamma,L_{r_2})^{\text{ord}}(2-r_2-\bm{\kappa}_f^{1/2})\bigr)
\]
by interchanging the second and third factors. Then, we have that
\[
s_N\circ (1\otimes w_N\otimes w_N)=(1\otimes w_N\otimes w_N)\circ s_N,
\]
and, taking into account the definition of $\mathbf{Det}$ and the fact that the K\"unneth isomorphism
\[
H^3_{\et}(Y_{\overline{\bb{Q}}}^3,\bm{\cl{A}}_{r_1}'\boxtimes \bm{\cl{A}}_{r_2}\boxtimes\bm{\cl{A}}_{r_2})\cong H^1_{\et}(Y_{\overline{\bb{Q}}},\bm{\cl{A}}_{r_1}')\otimes H^1_{\et}(Y_{\overline{\bb{Q}}},\bm{\cl{A}}_{r_2})\otimes H^1_{\et}(Y_{\overline{\bb{Q}}},\bm{\cl{A}}_{r_2})
\]
is given by cup-product, which is anti-commutative in degree 1 (\emph{cf.} the proof of \cite[Prop.~4.1.2]{LZ}), we deduce that $s_N(\kappa^{(2)})=(-1)^{r_1/2+r_2+1}\kappa^{(2)}$. Define $\tilde{s}_N=(-N)^{-r_2} s_N\circ (1\otimes w_N\otimes w_N)$. Then, we have that
\[
\tilde{s}_N(\kappa^{(2)})=(-1)^{r_1/2+1}\kappa_f^{-1/2}(N)([p]'_N w_N\otimes 1\otimes 1)\kappa^{(2)}.
\]
Since $(1\otimes \pi_{1\ast}\otimes \pi_{2\ast})\circ \tilde{s}_N=\tilde{s}_{N_g}\circ (1\otimes\pi_{1\ast}\otimes\pi_{2\ast})$, it follows that
\[
\tilde{s}_{N_g}(\kappa^{(3)})=-[N]^{-1/2}([p]'_N w_N\otimes 1\otimes 1)\kappa^{(3)}
\]
and therefore that
\[
\tilde{s}_{N_g}\bigl(\kappa^{(3)}(\hf,g,g^\ast)\bigr)=-[N]^{-1/2}(w_N\otimes 1\otimes 1)\kappa^{(3)}(\hf,g,g^\ast).
\]
Finally, it follows from \cite[Prop. 2.3.6]{How2} that $-[N]^{-1/2}w_N$ acts on $\bb{V}_\hf$ as multiplication by $\varepsilon(f)$, so the last part of the proposition follows from the previous lemma.
\end{proof}

\begin{remark}
From the definition of the map
\[
\frk{Log}(\hf,g_\alpha,g^\ast_\alpha):H^1_{\bal}\left(\bb{Q}_p,\bb{V}(\hf,g,g^\ast)\right)\longrightarrow \Lambda_\hf[1/p],
\]
one can see that it factors through the cohomology of $\bb{V}(\hf,g,g^\ast)_f$, and therefore that it factors through
\[
H^1_{\bal}\left(\bb{Q}_p,\bb{V}(\hf,g,g^\ast)\right)\longrightarrow H^1\bigl(\bb{Q}_p,\bb{V}_\hf(-\bm{\kappa}_f^{1/2})\otimes\ad^0(T_g)\bigr)\longrightarrow H^1\left(\bb{Q}_p,\bb{V}(\hf,g,g^\ast)_f\right).
\]
Therefore, \emph{without} the need to appeal to the reciprocity law, it follows from Proposition~\ref{prop:sign-behaviour} that when $\varepsilon(f)=-1$ we have
\[
\frk{Log}(\hf,g_\alpha,g^\ast_\alpha)(\kappa)=0.
\]
Of course, this can also be seen from the reciprocity law: Since $\varepsilon(f)=-1$ forces the vanishing of $L(\hf_{k'},k'/2)$ for all $k'\equiv k\pmod{2(p-1)}$, and this is a factor of $L(\hf_{k'}\otimes g\otimes g^\ast, c')$, it follows form the interpolation formula that $\mathscr{L}_p(\breve{\hf},\breve{g},\breve{h})$ is identically zero.
\end{remark}

\begin{remark}
As noted above, the discussion in this section is unnecessary for the applications that we will discuss. Indeed, as observed in the previous remark, the reciprocity law factors through  $H^1\bigl(\bb{Q}_p,\bb{V}_\hf(-\bm{\kappa}_f^{1/2})\otimes\ad^0(T_g)\bigr)$.
Therefore,
%if the triple product $L$-function does not vanish at some point in the range of interpolation,
the nonvanishing of the triple product $p$-adic $L$-function at some point (necessarily when $\varepsilon(f)=+1$) implies that the image of $\kappa^{(3)}$ in this group is nontrivial, which is what we will actually need. However, it is interesting that %, at least in some cases,
we can already see from the geometric construction that the class lies where it is expected.
\end{remark}

\begin{remark}
Let us discuss the sign in a little bit more detail. In order to construct the $p$-adic $L$-function attached to $(f,g,g^\ast)$, it is required in \cite{Hs} that the local signs at finite primes of the arithmetic specializations of the representation $\bb{V}(\hf,\hg,\hg^\ast)=\bb{V}_{\hf} \hat{\otimes}\bb{V}_{\hg} \hat{\otimes} \bb{V}_{\hg^\ast}(-1-\bm{\kappa}^\ast)$ are all equal to 1. In particular, in our case, this imposes the condition that $\varepsilon_\ell(\hf_{k'})=\varepsilon_\ell(\hf_{k'}\otimes\ad^0(g))$ for all $\ell\mid N$ and for all $k'\equiv k\pmod{2(p-1)}$. The corresponding signs at infinity can be computed from the Hodge types $\{(p,q), (q,p)\}$ as in \cite[\S5.3]{deligne}. For the representation $V_{\hf_{k'}}\otimes\ad^0(V_g)$, the Hodge types are as follows:
\begin{enumerate}
\item[(i)] $\{(k'/2+l-2,-k'/2-l+1), (-k'/2-l+1,k'/2+l-2)\}$;
\item[(ii)] $\{(k'/2-1,-k'/2), (-k'/2,k'/2-1)\}$;
\item[(iii)] $\{(k'/2-l,-k'/2+l-1), (-k'/2+l-1,k'/2-l)\}$.
\end{enumerate}
After that, and following the results of [\emph{loc.\,cit.}], we get that the sign $\varepsilon_\infty(f\otimes\ad^0(g))$ is $(-1)^{k'/2}$ if $k' \geq 2l$ and $(-1)^{1+k'/2}$ if $k'<2l$. The sign of $\varepsilon_\infty(\hf_{k'})$, however, is always equal to $(-1)^{k'/2}$. Therefore, in the balanced region (i.e. for $k'<2l$), the motives attached to $\hf_{k'}$ and $\hf_{k'}\otimes\ad^0(g)$ have opposite global signs. Since it is in this region that the corresponding specializations of the class $\kappa^{(3)}(\hf,g,g^\ast)$ belong to the Bloch-Kato Selmer group, we expect the behaviour that was shown in Proposition~\ref{prop:sign-behaviour}.
\end{remark}

\section{The $p$-adic $L$-function}\label{sec:Lp}

%Let $g\in S_l(N_g,\chi_g)$ and $\psi$ a Hecke character of $K$ of infinity type $(1-k,0)$ as introduced in $\S\ref{sec:Selmer}$, and recall that we assume that $(p)=\fp\overline{\fp}$ splits in $K$.

In this section, we keep the assumption that $(p)=\fp\overline{\fp}$ splits in $K$. In addition, from now on, for simplicity we assume that $p\nmid h_K$, where $h_K$ is the class number of $K$.

Let $g\in S_l(N_g,\chi_g)$ be an ordinary newform not of CM-type. Let $\mathfrak{c}$ be an ideal of $\cO_K$ coprime to $p$, and fix a Hecke character $\psi_0$ of infinity type $(-1,0)$ and conductor $\mathfrak{c}\fp^e$ with $e\in\{0,1\}$. We assume that the central character $\varepsilon_{\psi_0}$ of $\psi_0$ is of the form
\begin{equation}\label{eq:central}
\textrm{$\varepsilon_{\psi_0}=\varepsilon_K\omega^{r_1}$ for some even integer $r_1$},\tag{H0}
\end{equation}
where $\omega$ is the Teichm\"uller character.
%(i.e. the conductor is $\frk{c}$ if $p-1\mid r_1$ and $\frk{cp}$ otherwise).

%Assume that the prime $p$ splits in $K$ and let $\frk{p}$ be the prime of $K$ above $p$ determined by the embedding $\iota_p:\overline{\Q}\hookrightarrow\overline{\Q}_p$ that we fixed in the introduction. We also assume that $p$ does not divide the class number $h_K$.

\subsection{Lifting of automorphic representations}

Let $\pi$ be the cuspidal automorphic representation of ${\rm GL}_2(\bb{A}_{\bb{Q}})$ attached to $g$. The central character of $\pi$ is the adelic character $\omega_g$ defined by the condition that for any prime $q\nmid N_g$ and any uniformizer $\varpi_q$ we have $\omega_{g,p}(\varpi_q)=\chi_g(q)$. Since $p\nmid N_g$, the local component $\pi_p$ is a spherical representation, and it follows from \cite[Thm.~4.6.4]{Bump} and its proof that $\pi_p$ is isomorphic to the principal series $\pi(\chi,\chi^{-1}\omega_g)$, where $\chi$ is the unramified character of $\bb{Q}_p^\times$ defined by $\chi(p)=\alpha_p(g)p^{(1-l)/2}$.

Since we are assuming that $g$ is not of CM-type, and in particular it does not have CM by $K$, it follows from \cite[Prop. 2.3.3]{GJ} that $\pi$ admits, adopting the terminology of [\emph{op.\,cit.}], a base change lifting to a cuspidal automorphic representation $\pi_K$ of ${\rm GL}_2(\bb{A}_K)$. We fix such a lifting. Observe that if $\frk{p},\overline{\frk{p}}$ are the places of $K$ above $p$, then $\pi_{K,\frk{p}}\simeq \pi_{K,\overline{\frk{p}}}\simeq \pi_p$.

From the assumption that $g$ is not of CM-type we deduce that there is no non-trivial character $\eta$ of $K^\times\backslash\bb{A}_K^\times$ such that $\pi_K\simeq\pi_K\otimes \eta$. Indeed, the existence of such a character would imply that there exists a quadratic extension $L$ of $K$ such that, for all prime $\ell$, the restriction to $G_K$ of the $\ell$-adic Galois representation attached to $g$ is induced from a character of $G_L$, which is not possible by \cite[Thm.~2.1]{Rib85}. Now, it follows from \cite[Thm.~9.3]{GJ} that $\pi_K$ admits an adjoint lifting to a cuspidal automorphic representation $\Pi_{{\rm Ad}^0(g)}$ of ${\rm GL}_3(\bb{A}_K)$. Fix such a lifting and define
%the cuspidal automorphic representation
\[
\Pi:= \Pi_{{\rm Ad}^0(g)}\otimes \psi_0\vert\cdot\vert^{1/2}.
\]
Observe that $\Pi_\frk{p}\simeq \Pi_{\overline{\frk{p}}}\simeq \pi(\chi^2\omega_{g,p}^{-1},1,\chi^{-2}\omega_{g,p})\otimes  \psi_0\vert\cdot\vert^{1/2}$ and it follows from the definition of $\chi$ that $\chi^2\omega_{g,p}^{-1}\neq \vert\cdot\vert^{\pm 1/2}$ and therefore that $\pi(\chi^2\omega_{g,p}^{-1},1,\chi^{-2}\omega_{g,p})=\Ind_{B}^{\mathrm{GL}_3}(\chi^2\omega_{g,p}^{-1},1,\chi^{-2}\omega_{g,p})$, where $B$ denotes the Borel subgroup of $\GL_3(\mathbb Q_p)$.

\subsection{Descent to unitary groups}

Let $U(2,1)$ be the quasi-split unitary group corresponding to the quadratic extension $K/\bb{Q}$. Let $\Phi\in\mathrm{GL}_3(K)$ be the matrix whose entries are $\Phi_{ij}=(-1)^{i-1}\delta_{i,4-j}$. Then we can describe $U(2,1)$ by specifying its functor of points:
\[
U(2,1)(R)=\left\lbrace g\in\mathrm{GL}_3(R\otimes_{\bb{Q}} K)\;\colon\; g\Phi\,{}^t\!\overline{g}=\Phi\right\rbrace
\]
for any $\bb{Q}$-algebra $R$.

Let $U(3)$ be the definite unitary group whose functor of points is given by
\[
U(3)(R)=\left\lbrace g\in\mathrm{GL}_3(R\otimes_{\bb{Q}} K)\;\colon\; g\,{}^t\!\overline{g}=I_3\right\rbrace.
\]

Given a representation $\rho$ of ${\rm GL}_3(\bb{A}_K)$, let $\tilde{\rho}$ be the representation defined on the same space by $\tilde{\rho}(x)=\rho({}^t\!\bar{x}^{-1})$. Then, the representation $\Pi$ defined above satisfies $
\Pi\simeq \tilde{\Pi}$, and so it follows from \cite[Thm. 13.3.3]{Rog} that there exists a cuspidal automorphic representation $\sigma'$ of $U(2,1)(\bb{A}_\bb{Q})$ whose base change to $K$ is isomorphic to $\Pi$. Fix such a representation $\sigma'$. Observe that $\sigma'_p\simeq\pi(\chi^2\omega_g^{-1},1,\chi^{-2}\omega_g)\otimes  \psi_0\vert\cdot\vert^{1/2}$ under the identification $U(2,1)(\bb{Q}_p)=\mathrm{GL}_3(\bb{Q}_p)$. Also, from [\emph{op.\,cit.}, Prop. 13.2.2], the local representation $\sigma'_\infty$ is square-integrable, so, applying [\emph{op.\,cit.}, Prop. 14.6.2], we can transfer $\sigma'$ to a representation $\sigma$ of $U(3)$. The local components of $\sigma$ at finite primes agree with those of $\sigma'$, so in particular we have that $\sigma_p\simeq \sigma_p'$.

\begin{remark}\label{rk:fromU3toGU3}
Let $GU(3)$ be the definite unitary similitude group whose functor of points is given by
\[
GU(3)(R)=\left\lbrace g\in\mathrm{GL}_3(R\otimes_{\bb{Q}} K)\;\colon\; g\,{}^t\!\overline{g}=\nu(g) I_3\text{ for some } \nu(g)\in R^\times\right\rbrace.
\]
As explained in \cite[\S1.8]{BR93}, one can extend $\sigma$ to an irreducible automorphic representation of $GU(3)$ by choosing an extension of the central character of $\sigma$ to the center of $GU(3)$.
\end{remark}

\begin{comment}

\subsection{Transfer to unitary groups}

Our approach to construct a $p$-adic $L$-function for the symmetric square $\Ad^0(g)_{/K}$ is based on using a base transfer to the unitary group $U(3)$, and then use the general theory available in this setting. We use the following result due to Rogawski \cite{Rog}.

\begin{proposition}
A Galois representation for $\GL(3)_{/K}$, invariant under the automorphism $g \mapsto ^t\bar{g}^{-1}$, comes from a base change transfer of $U(3)$.
\end{proposition}

\begin{remark}
Rogawski's result is more precise, establishing a bijection between representations for $\GL(3)_{/K}$ invariant under the automorphism $g \mapsto ^t\bar{g}^{-1}$ and the so-called stable representations of $U(3)$. See loc.\,cit. for the precise definitions.
\end{remark}

\end{comment}

\subsection{$p$-adic $L$-functions for unitary groups}

A construction of $p$-adic $L$-functions for unitary groups is given in \cite{EW}, and, in great generality in \cite{EHLS}. Here we deduce from these works the existence of an anticyclotomic $p$-adic $L$-function for the conjugate self-dual representation $V$ in $\S\ref{subsec:adjoint}$.

%In this case, we use the approach of \cite{EW}, who prove a precise interpolation formula in the context that we need, up to some local constants which are explicitly computed in the work of Eischen--Liu \cite{EL}.

Let $\mathcal{O}$ be the ring of integers of a finite extension of $\Q_p$ containing the values of $\psi_0$, and write
\[
\Lambda^{\rm ac}=\cO\dBr{\Gamma^{\rm ac}}
\]
for the anticyclotomic Iwasawa algebra, where $\Gamma^{\rm ac}$ is the Galois group of the anticyclotomic $\Z_p$-extension of $K$.

We will need to consider the following CM periods, as they are introduced in \cite{BDP12}:
\begin{itemize}
\item $\Omega_{\infty} \in \mathbb{C}^\times$ is the complex period attached to $K$ defined in [\emph{op.\,cit.}, eq.~(2-15)];
\item $\Omega_p \in \bb{C}_p^\times$ is the $p$-adic period attached to $K$ defined in [\emph{op.\,cit.}, eq.~(2-17)].
\end{itemize}

%{\color{blue} Canviem aqui la notacio doncs per fer-la consistent amb la del triple product (de la manera que surt a la factoritzacio)?}

\begin{theorem}\label{thm:sym}

There exists an element
\[
L_p({\rm ad}^0(g_K)\otimes\psi_0)\in{\rm Frac}\,\Lambda^{\rm ac}
\]
such that for all characters $\xi$ of $\Gamma^{\rm ac}$ crystalline at both $\fp$ and $\bar{\fp}$ and corresponding to a Hecke character of infinity type $(-n,n)$ with $n\equiv r_1/2\pmod{p-1}$ and $n\geq l-1$, we have

%There exists an element \[ L_p(\hf, \ad^0(g)) \in \Frac(\Lambda_{\hf}) \] such that for all weights $k'\geq 2l$ with $k'\equiv k \pmod{2(p-1)}$ we have

\[
L_p({\rm ad}^0(g_K)\otimes\psi_0)(\xi) = \biggl( \frac{\Omega_p}{\Omega_{\infty}} \biggr)^{6n+3} \cdot \pi^{3n} \cdot \Gamma(n,l) \cdot \mathcal{E}_p({\rm ad}^0(g),\psi_0\xi)^2 \cdot L({\rm ad}^0(g_K) \otimes\psi_0^{-1}\xi^{-1}\omega^n,0),
\]
where:
\begin{itemize}
	\item $\Gamma(n,l)=(n+l-1)!\cdot n!\cdot(n-l+1)!$,
	\item $\mathcal{E}_p({\rm ad}^0(g),\psi_0\xi)=\left(1-\frac{\alpha_g(\psi_0\xi\omega^{-n})(\fp)}{\beta_g p}\right) \cdot\left(1-\frac{(\psi_0\xi\omega^{-n})(\fp)}{p}\right)\cdot\left(1-\frac{\beta_g (\psi_0\xi\omega^{-n})(\fp)}{\alpha_g p}\right)$.

%\item $\mathcal{E}(\hf_{k'},\ad^0(g))=(1-\frac{\alpha_{g}\beta_{k'}}{\beta_g p^{k'/2}}) (1-\frac{\beta_{k'}}{p^{k'/2}}) (1-\frac{\beta_{g}\beta_{k'}}{\alpha_g p^{k'/2}})$,
%\item $\hf_{k'}^\sharp$ is the newform attached to the $p$-stabilized newform $\hf_{k'}$.
\end{itemize}
\end{theorem}

\begin{proof}
Let $\sigma$ be the irreducible automorphic representation of $U(3)$ introduced in the previous subsection. Let $\Sigma$ be the set of places of $\bb{Q}$ consisting of $p$, infinity, the primes dividing $D_K$, and the primes at which $\sigma$ ramifies. On account of Remark~\ref{rk:fromU3toGU3}, the main result of \cite{EW} yields an element $\mathscr{L}^\Sigma_p\in\Lambda^{\ac}[1/p]$ such that, for all $\xi$ as in the statement, satisfies
\[
\mathscr{L}^\Sigma_p(\xi)=\biggl( \frac{\pi\Omega_p}{\Omega_{\infty}} \biggr)^{6n+3}\cdot\mathcal{E}_{p}(\xi)\cdot \mathcal{E}_\infty(\xi) \cdot L^\Sigma(\tilde{\sigma},\xi^{-1}\omega^n,0),
\]
%where $n\geq l-1$ is an integer such that the infinity type of $\xi$ is $(-n,n)$,
where $\tilde{\sigma}$ is the contragredient of $\sigma$, and $\cl{E}_p(\xi)$ and $\cl{E}_\infty(\xi)$ are certain modified Euler factors at $p$ and $\infty$, respectively. Since $\Sigma$ contains $p$, infinity, and all the ramified primes, we have that
\[
L^\Sigma(\tilde{\sigma},\xi^{-1}\omega^n,0)=L^\Sigma({\rm ad}^0(g_K) \otimes\psi_0^{-1}\xi^{-1}\omega^n,0).
\]

Since we are assuming that $p$ splits in $K$, the form of the modified Euler factor at $p$ can be extracted from \cite[eq.~(86)]{EHLS}. Up to a nonzero rational factor independent of $\xi$, it is given by
\[
\cl{E}_p(\xi)=\frac{\mathcal{E}_p({\rm ad}^0(g),\psi_0\xi)^2}{L_p(\ad^0(g_K)\otimes\psi_0^{-1}\xi^{-1}\omega^n,0)}.
\]

The form of the modified Euler factor at infinity can be extracted from \cite[eq.~(2.3.1)]{EL}. This formula, with $a=3$, $b=0$, $\underline{\tau}=(l,0,-l)$, $r=2n+2$ and $s=0$,
%combined with the formula for the factor denoted by $A(\pi_\infty,\chi_{u,\infty})$ in [\emph{op.\,cit.}, \S1.3],
yields, up to a nonzero rational factor independent of $\xi$,
\[
\cl{E}_\infty(\xi)=(2\pi i)^{-3n-3}\cdot\Gamma(n,l).
\]

Finally, the Euler factors at primes $\ell\in\Sigma\setminus\lbrace p,\infty\rbrace$ can be $p$-adically interpolated by certain elements $\mathcal{P}_\ell\in\Lambda^{\rm ac}$, and, multiplying by their inverses, we obtain the $p$-adic $L$-function in the statement.
\end{proof}

%Furthermore, although the interpolation formulae in \cite{EW} are computed for $k_1=0$, one can get the general case by using the results on the computation of the doubling archimedean zeta integrals in \cite{EL}.

\subsection{CM Hida family}\label{subsec:CM-Hida}

Let $\Gamma_K$ be the Galois group of the $\Z_p^2$-extension $K_\infty/K$ and put
\[
\Gamma_\fp={\rm Gal}(K_{\fp^\infty}/K)\simeq\Z_p,
\]
where $K_{\fp^\infty}$ is the maximal subfield of $K_\infty$ unramified outside of $\fp$, so that $K_{\fp^\infty}$ is the $\Z_p$-extension of $K$ inside the ray class field $K(\fp^\infty)$. Since we are assuming that $p\nmid h_K$, viewing $1+p\Z_p$ as a subgroup of $\cO_{K,\fp}^\times$, the restriction of the (geometrically normalized) Artin map to $K_\fp^\times$ induces an isomorphism ${\rm art}_\fp:1+p\Z_p\simeq\Gamma_{\fp}$. Let $\gamma_\fp\in\Gamma_{\fp}$ be the topological generator corresponding to $1+p$ under this isomorphism and, for the variable $S$, let $\Psi_S:\Gamma_K\rightarrow\Z_p\llbracket{S}\rrbracket^\times$ be the character given by
\[
\Psi_S(\sigma)=(1+S)^{l(\sigma)},
\]
where $l(\sigma)\in\Z_p$ is defined by $\sigma\vert_{K_{\fp^\infty}}=\gamma_\fp^{l(\sigma)}$. Consider the formal $q$-expansion
\begin{equation}\label{eq:CM-F}
\CM(S)(q):=\sum_{(\mathfrak{a},\fp\mathfrak{c})=1}\psi_0(\sigma_{\mathfrak{a}})\Psi^{-1}_S(\sigma_\mathfrak{a})q^{\mathbf{N}(\mathfrak{a})}\in\mathcal{O}\llbracket{S}\rrbracket\llbracket{q}\rrbracket,
\end{equation}
where $\sigma_\mathfrak{a}\in{\rm Gal}(K(\mathfrak{c}\fp^\infty)/K)$ is the Artin symbol of $\mathfrak{a}$. Then, for every $k\geq 2$, the specialization of $\CM$ at $S=(1+p)^{k-2}-1$ is given by the theta series
\[
\hf_k=\sum_{(\mathfrak{a},\mathfrak{pc})=1}\psi_0(\mathfrak{a})\lambda^{k-2}(\mathfrak{a})q^{\mathbf{N}(\mathfrak{a})}\in S_k(D_K\mathbf{N}(\mathfrak{c})p,\omega^{2+r_1-k}),
\]
where $\lambda$ is the unique (since $p\nmid h_K$) Hecke character of infinity type $(-1,0)$ and conductor $\fp$ whose $p$-adic avatar factors through $\Gamma_\fp$. In particular, $\hf_2$ is the ordinary $p$-stabilization of $\theta_{\psi_0}$.

\begin{remark}\label{rem:psi_0}
If $\psi$ is a Hecke character of infinity type $(1-k,0)$ as in $\S\ref{subsec:adjoint}$, then $\psi_0:=\psi\lambda^{2-k}$ is a Hecke character as above (in particular, satisfying (\ref{eq:central}) with e.g.  $r_1=k-2$), and so the  resulting $\hf_k$ recovers the $p$-stabilization of $\theta_\psi$. From now on we shall always assume that $\psi$ and $\psi_0$ are related in this manner, and refer to $\hf=\boldsymbol{\theta}_{\psi_0}$ as the CM Hida family attached to $\psi$ (or $\psi_0$).
\end{remark}

\subsection{A factorization formula}\label{subsec:factor}

In this section we prove a factorization formula relating the $p$-adic $L$-function attached to $V$ in Theorem~\ref{thm:sym} to anticyclotomic $p$-adic $L$-functions attached to the other two representations in the decomposition (\ref{eq:dec-V}).

Put $N={\rm lcm}(N_g,N_\psi)$, where $N_{\psi}:=D_K\mathbf{N}(\mathfrak{c})$. In addition to the previous hypotheses, from now on we shall also assume that:

%We now recall the other $p$-adic $L$-functions involved in the picture, beginning with the triple product $p$-adic $L$-function of Hsieh \cite{Hs}. We make the following assumptions (which add to the ones already in place):
\begin{itemize}
\item[{\rm (a)}] %for some integer $k \geq 2l$ with $k\equiv r_1+2 \pmod{2(p-1)}$ we have \[\varepsilon_\ell\left(V_{\hf_k}\otimes\ad(V_g)(1-k/2)\right)=+1\] for all primes $\ell\mid N={\rm lcm}(N_g,N_\psi)$;
$\varepsilon_\ell(V_{fgg^{\ast}})=+1$ for all primes $\ell\mid N$,
\item[{\rm (b)}] $\mathrm{gcd}(N_g,N_\psi)$ is squarefree.
\end{itemize}
With notations as in Remark~\ref{rem:shapiro}, here $\varepsilon_\ell(V_{fgg^{\ast}})$ denotes the epsilon-factor of the Weil--Deligne representation attached to the restriction of  $V_{fgg^\ast}$ to $G_{\Q_\ell}$.

Note that it follows from (\ref{eq:central}) that the Galois representation of the Hida family $\hf=\boldsymbol{\theta}_{\psi_0}$ attached to $\psi$ is residually irreducible and $p$-distinguished (see also \cite[Rem.~5.1.3]{LLZ}). For the following result, we adopt the definition of congruence ideal in \cite[\S3.3]{Hs}.

\begin{theorem}\label{thm:hsieh}
Under the above hypotheses, there exists an element
\[
\mathscr{L}_p({\rm ad}(g_K)\otimes\psi_0)\in{\rm Frac}\,\Lambda^{\rm ac}
\]
such that for all characters $\xi$ of $\Gamma^{\rm ac}$ crystalline at both $\fp$ and $\bar{\fp}$ and corresponding to a Hecke character of infinity type $(-n,n)$ with $n\equiv r_1/2\pmod{p-1}$ and $n\geq l-1$, we have
\[
\mathscr{L}_p({\rm ad}(g_K)\otimes\psi_0)(\xi)^2 =\Gamma(n,l,l)\cdot\frac{\mathcal{E}_p({\rm ad}(g),\psi_0\xi)^2} {\mathcal{E}_0(\psi_0\xi)^2\cdot\mathcal{E}_1(\psi_0\xi)^2}\cdot\prod_{\ell\mid N}\tau_\ell\cdot\frac{L({\rm ad}(g_K)\otimes\psi_0^{-1}\xi^{-1}\omega^{n},0)}{(2\pi i)^{4n+4}\cdot\langle\theta_{\psi_0\xi_n},\theta_{\psi_0\xi_n}\rangle^2},
\]
where:
\begin{itemize}
    %\item $c'=(k'+2l-2)/2$,
	\item $\Gamma(n,l,l)=(n+l-1)!\cdot (n!)^2\cdot(n-l+1)!$,
	\item $\mathcal{E}_p({\rm ad}(g),\psi_0\xi)=\left(1-\frac{\alpha_g(\psi_0\xi\omega^{-n})(\fp)}{\beta_g p}\right) \cdot\left(1-\frac{(\psi_0\xi\omega^{-n})(\fp)}{p}\right)^2\cdot\left(1-\frac{\beta_g (\psi_0\xi\omega^{-n})(\fp)}{\alpha_g p}\right)$,
\item
$\mathcal{E}_0(\psi_0\xi)=\left(1-\frac{(\psi_0\xi\omega^{-n})(\fp)}{(\psi_0\xi\omega^{-n})(\bar\fp)}\right)$,
$\mathcal{E}_1(\psi_0\xi)=\left(1-\frac{(\psi_0\xi\omega^{-n})(\fp)}{p(\psi_0\xi\omega^{-n})(\bar\fp)}\right)$,
\item $\tau_\ell$ is an explicit nonzero rational number independent of $n$,
\item $\theta_{\psi_0\xi_n}$ is the theta series of weight $2n+2$ attached to $\psi_0\xi_n:=\psi_0\xi\omega^{-n}\vert\cdot\vert^{-n}$.
\end{itemize}
Moreover, if $\mathcal{H}$ is any generator of the congruence ideal of $\boldsymbol{\theta}_{\psi_0}$, then  $\mathcal{H}\cdot\mathscr{L}_p({\rm ad}(g_K)\otimes\psi_0)$ belongs to $\Lambda^{\rm ac}$.
\end{theorem}

\begin{proof}
This is essentially a reformulation of \cite[Thm.~A]{Hs} specialized to our setting. Let $\hf=\CM$ be the Hida family attached to the Hecke character $\psi_0$ as in (\ref{eq:CM-F}), with associated big Galois representation $\mathbb{V}_{\hf}$, and denote by $\mathbb{V}(\hf,g,g^*)$ the Kummer self-dual twist of the triple tensor product $\mathbb{V}_{\hf}\hat\otimes_{\mathcal{O}}T_g\otimes_{\mathcal{O}}T_{g^*}$ introduced in \cite[\S{7.1}]{ACR} (and recalled in $\S\ref{section:bottomclass}$ above).
Since $\mathbb{V}_{\hf}\simeq{\rm Ind}_K^\Q(\psi_0^{-1}\Psi_{S})$, we immediately find that
\begin{align*}
\mathbb{V}(\hf,g,g^*)
%&=({\rm Ind}_K^\Q\psi_0\Psi_S^{-1})\hat{\otimes}_{\mathcal{O}}T_g\otimes_{\mathcal{O}}T_{g^\ast}(1-l)\otimes(\Psi_S^{1/2}\circ\mathscr{V}),\\
&\simeq{\rm ad}(T_g)\otimes{\rm Ind}_K^{\Q}(\psi_0^{-1}\omega^{r_1/2}\Psi_S^{(1-\tau)/2}),
\end{align*}
where for a character $\chi$ of $G_K$ we denote by $\chi^\tau$ the composition of $\chi$ with the action of the non-trivial automorphism $\tau$ of $K/\Q$, and put $\chi^{1-\tau}:=\chi(\chi^\tau)^{-1}$.
%where $\mathscr{V}:G_\Q^{\rm ab}\rightarrow G_K^{\rm ab}$ is the transfer map.

By \cite[Thm.~A]{Hs}, attached to $(\hf,g,g^\ast)$ (and a specific choice of level-$N$ test vectors for this triple), there is an ``unbalanced'' triple product $p$-adic $L$-function $\mathscr{L}_p(\hf,g,g^*)\in{\rm Frac}\,\mathcal{O}\llbracket\Gamma_\fp\rrbracket$ interpolating, for all $k'\equiv r_1+2\pmod{2(p-1)}$ with $k'\geq 2l$, the (central) values at $s=0$ of the triple product $L$-function
\[
L(\mathbb{V}(\hf_k,g,g^*),s)=L({\rm ad}(g_K)\otimes\psi_0^{-1}\xi^{-1}\omega^{r_1/2},s).
\]
where we put $\xi$ to denote the specialization of $\Psi_{S}^{(\tau-1)/2}$ at $S=(1+p)^{k'-2}-1$, so $\xi^{-1}$ is a character of $\Gamma^{\rm ac}$ crystalline at both $\fp$ and $\bar{\fp}$ corresponding to a Hecke character of infinity type $(-(k'/2-1),k'/2-1)$. Taking $\mathscr{L}_p({\rm ad}(g_K)\otimes\psi_0)$ to be the image of $
\mathscr{L}_p(\hf,g,g^\ast)$ under the
map ${\rm Frac}\,\mathcal{O}\llbracket\Gamma_\fp\rrbracket\rightarrow{\rm Frac}\,\Lambda^{\rm ac}$ determined by $\gamma_\fp\mapsto\gamma_\fp^{\tau-1}$, we thus see that the result follows from \cite[Thm.~A]{Hs}.
\end{proof}

%From now, we put
%\[
%L_p(\hf,g,g^\ast):=\mathscr{L}_p(\breve{\hf}, \breve{g}, \breve{g}^\ast)^2\nonumber
%\]
%for the test vectors $(\breve{\hf},\breve{g},\breve{g}^\ast)$ of level $N$ provided by Theorem~\ref{thm:hsieh}.

%Let $\mathfrak c \subset \mathcal O_K$ be an integral ideal of the imaginary quadratic field $K$.

We next discuss an anticyclotomic $p$-adic $L$-function associated with $V'$, arising from a suitable restriction of Katz's $p$-adic $L$-function.

Denote by $\Sigma$ the set of algebraic Hecke characters $\xi$ of $K$ for which $s=0$ is a critical point for $L(\xi,s)$ in the sense of Deligne. This set can be written as the disjoint union $\Sigma = \Sigma_{\fp} \cup \Sigma_{\bar{\fp}}$, where
\begin{align*}
 \Sigma_\fp &= \{ \xi \in \Sigma \text{ of infinity type } (a,b), \text{ with } a \geq 1,\, b \leq 0 \},\\
 \Sigma_{\bar{\fp}} &= \{ \xi \in \Sigma \text{ of infinity type } (a,b), \text{ with } a \leq 0,\, b \geq 1 \}.
\end{align*}
Note that the involution $\xi\mapsto\xi^\tau$
%, where $\xi'$ is the composition of $\xi$ with complex conjugation,
takes characters in $\Sigma_\fp$ to characters in $\Sigma_{\bar\fp}$, and vice versa.
%interchanges the regions $\Sigma_{\rm crit}^{(1)}$ and $\Sigma_{\rm crit}^{(2)}$.

%Thus $\xi \in \Sigma_K^{\text{crit}}$ if and only if $s=0$ is a critical point for $L(\xi,s)$. As an extra piece of notation, let $\hat{\Sigma}_K$ denote the completion of $\Sigma_K^{(2)}$ with respect to the compact open topology of $\mathcal O_{L_p}$-valued functions of a certain subset of $\mathbb A_K^{\times}$, as described in \cite[\S5.2]{BDP} (the convention that we use for the infinity type of a Hecke character is the opposite to the one in [\emph{op.\,cit.}]).

Let $G_\mathfrak{c}={\rm Gal}(K(\mathfrak{c}p^\infty)/K)$ be the Galois group of the ray class field of $K$ of conductor $\mathfrak{c}p^\infty$, and denote by $\Z_p^{\rm ur}$ the completion of the ring of integers of the maximal unramified extension of $\Q_p$. The following result is originally due to Katz.

\begin{theorem}\label{thm:katz}
There exists an element $\mathscr{L}^{\rm Katz}_{\fp,\mathfrak{c}}(K)\in\Z_p^{\rm ur}\llbracket{G}_{\mathfrak{c}}\rrbracket$ uniquely characterized by the property that for every character of $\Gamma_\mathfrak{c}$ corresponding to a Hecke character $\xi\in\Sigma_\fp$ of infinity type $(k_1,k_2)$ and conductor dividing $\mathfrak{c}$ we have
\[
\mathscr{L}_{\fp,\mathfrak{c}}^{\rm Katz}(K)({\xi})=\biggl(\frac{\Omega_p}{\Omega_\infty}\biggr)^{k_1-k_2}(k_1-1)!\cdot\biggl(\frac{\sqrt{D_K}}{2\pi}\biggr)^{k_2}\cdot(1-p^{-1}\xi^{-1}({\fp})p^{-1})(1-\xi(\bar{\fp}))\cdot L_\mathfrak{c}(\xi,0),
\]
where $L_{\mathfrak{c}}(\xi,s)$ is the $L$-function of $\xi$ with the Euler factors at the primes $\mathfrak{l}\vert\mathfrak{c}$ removed.
%Similarly, there exists an element $L_{\fp}^{\rm Katz}(K)\in\Lambda_K^{\rm ur}$ uniquely characterized by the property that for every character of $\Gamma_K$ corresponding to a Hecke character $\xi\in\Sigma_{\rm crit}^{(1)}$ of infinity type $(k_2,k_1)$ we have
%\[
%L_{\fp}^{\rm Katz}(K)({\xi})=\frac{\Omega_p^{k_1-k_2}}{\Omega_\infty^{k_1-k_2}}\cdot(k_1-1)!\cdot\biggl(\frac{\sqrt{D_K}}{2\pi}\biggr)^{k_2}\cdot(1-p^{-1}\xi^{-1}(\bar{\fp}))(1-\xi(\fp))\cdot L(\xi,0).
%\]
Moreover, we have the functional equation
\[
\mathscr{L}_{\fp,\mathfrak{c}}^{\rm Katz}(K)({\xi})=\mathscr{L}_{\fp,\bar{\mathfrak{c}}}^{\rm Katz}(K)(\xi^{-\tau}\mathbf{N}^{-1}),
\]
where the equality is up to a $p$-adic unit.
\end{theorem}

\begin{proof}
See \cite[Thm.~II.4.14]{deshalit} for a construction of $\mathscr{L}_{\fp,\mathfrak{c}}^{\rm Katz}(K)$ (corresponding to the measure on $G_\mathfrak{c}$ denoted by $\mu(\mathfrak{c}\bar{\fp}^\infty)$ in [\emph{loc.\,cit.}]), and \cite[Thm.~II.6.4]{deshalit} for the functional equation.
\end{proof}

Assume that $\mathfrak{c}$ is fixed under complex conjugation, i.e., $\bar{\mathfrak{c}}=\mathfrak{c}$. Denote by $\Delta_{\mathfrak{c}}$ the torsion subgroup of $G_\mathfrak{c}$, and put $\Gamma_K:=G_{\mathfrak{c}}/\Delta_\mathfrak{c}\simeq\Z_p^2$, which is identified with the Galois group of the unique $\Z_p^2$-extension $K_\infty/K$. We fix a decomposition %(non-canonical in general)
\begin{equation}\label{eq:factor-G}
G_{\mathfrak{c}}\simeq\Delta_{\mathfrak{c}}\times\Gamma_K.
\end{equation}
Put $\bar{\psi}_0=\psi_0\vert_{\Delta_\mathfrak{c}}$ and $\bar{\psi}_0^-=\bar{\psi}_0^{\tau-1}$, and note that the latter defines a finite order anticyclotomic Hecke character of conductor dividing $\mathfrak{c}p^s$ for some $s\geq 0$. Denote by $\mathscr{L}_{\fp,{\psi}_0}^{\rm Katz}(K)^-$ the image of $\mathscr{L}_{\fp,\mathfrak{c}}^{\rm Katz}(K)$ under the composite map
\[
\Z_p^{\rm ur}\llbracket{G_\mathfrak{c}}\rrbracket\rightarrow\Z_p^{\rm ur}\llbracket{\Gamma_K}\rrbracket\rightarrow\Z_p^{\rm ur}\llbracket\Gamma^\mathrm{ac}\rrbracket,
\]
where the first arrow is the projection defined by $\bar{\psi}_0^-$ and the second arrow is given by $\gamma\mapsto\gamma^{\tau-1}$ for $\gamma\in
\Gamma_K$.

From now on we shall assume that the above $\mathfrak{c}$ and $\psi_0$ satisfy the conditions (H1)--(H4) in the following result.
%\begin{propo}\label{katz-interpolation}
%There exists a $p$-adic analytic function
%\[ L_{\fp}^{\mathrm{Katz}}(K): \hat{\Sigma}_K \longrightarrow \mathbb C_p \]
%uniquely determined by the following interpolation property: if $\xi \in \Sigma^{(2)}_K$ is a character of infinity type $(k_1,k_2)$, then we have
%\[ \frac{L_{\fp}^{\mathrm{Katz}}(K)(\xi)}{\Omega_p^{k_1-k_2}} = \mathfrak a(\xi) \times \mathfrak e(\xi) \times \mathfrak f(\xi) \times \frac{L_{\mathfrak c}(\xi,0)}{\Omega^{k_1-k_2}}, \]
%with both sides lying in $\overline{\mathbb Q}$, where
%\begin{itemize}
%\item $\mathfrak a(\xi) = (k_1-1)! \pi^{-k_2}$,
%\item $\mathfrak e(\xi) = (1- p^{-1} \xi^{-1}(\fp))(1-\xi(\bar{\fp}))$,
%\item $\mathfrak f(\xi) = D_K^{k_2/2} 2^{-k_2}$,
%\item $L_{\mathfrak c}(\xi,s)$ is Hecke's $L$-function associated with $\xi$ with the Euler factors at primes dividing $\mathfrak c$ removed.
%\end{itemize}
%\end{propo}

\begin{proposition}\label{prop:congr}

In addition to {\rm (H0)}, assume that:
\begin{itemize}
\item[{\rm (H1)}] $\mathfrak{c}$ is only divisible by primes that are split in $K$;
\item[{\rm (H2)}] $\bar{\psi}_0^-$ has order prime-to-$p$ and the prime-to-$p$ part of its conductor it exactly $\mathfrak{c}$;
\item[{\rm (H3)}] $\bar{\psi}_0^-\vert_{G_{K_v}}\neq 1$ for all primes $v\vert p$ in $K$;
\item[{\rm (H4)}] $\bar{\psi}_0^-$ has order at least $3$.
\end{itemize}
Then, as an ideal of $\Z_p^{\rm ur}\dBr{\Gamma^{\rm ac}}$, the congruence ideal $C(\CM)$ is generated by
\[
h_K\cdot\mathscr{L}_{\fp,{\psi}_0}^{\rm Katz}(K)^-
\]
where $h_K$ is the class number of $K$.
\end{proposition}

\begin{proof}
A generator of $C(\CM)$ is given by a congruence power series $H(\CM)$ attached to $\CM$ as in \cite{hida-coates}. By our assumptions, this $H(\CM)$ corresponds to a branch character satisfying the hypotheses (1)--(4) in \cite[p.~466]{hida-coates}, so as noted in p.~469 of [\emph{op.\,cit.}], the result follows from the proof of the anticyclotomic Iwasawa main conjecture by Hida--Tilouine \cite{HT-ENS,HT-117} and Hida \cite{hida-coates}.
\end{proof}

\begin{definition}\label{triple:modified}
Put
\[
L_p({\rm ad}(g_K)\otimes\psi_0):=\bigl(\mathscr{L}_p({\rm ad}(g_K)\otimes\psi_0)\cdot h_K\cdot\mathscr{L}_{\fp,{\psi}_0}^{{\rm Katz}}(K)^-\bigr)^2,
\]
which by Theorem~\ref{thm:hsieh} and Proposition~\ref{prop:congr} defines an element in $\Z_p^{\rm ur}\dBr{\Gamma^{\rm ac}}$.
\end{definition}

We can now derive an anticyclotomic analogue of Dasgupta's factorization \cite[Thm.~1]{Das}, relating the $p$-adic $L$-function of Theorem~\ref{thm:hsieh} to the product of the $p$-adic $L$-functions in Theorem~\ref{thm:sym} and Theorem~\ref{thm:katz}. Similarly as in [\emph{loc.\,cit.}], %(and with notations as in the proof of Theorem~\ref{thm:hsieh}),
our result is a $p$-adic analogue of the factorization of complex $L$-functions
\[
L({\rm ad}(g_K)\otimes\chi,s)=L({\rm ad}^0(g_K)\otimes\chi,s)\cdot L(\chi,s)
\]
arising from the decomposition of $G_K$-representations
\[
{\rm ad}(V_g)\otimes\chi\simeq({\rm ad}^0(V_g)\otimes\chi)\oplus\chi.
\]
However, our proof is largely simplified by the fact that the three $p$-adic $L$-functions involved have a Zariski dense overlapping set of characters in the range of interpolation.

Our factorization formula will in fact involve the following anticyclotomic projection of the Katz $p$-adic $L$-function.

\begin{definition}\label{def:Katz-}
Viewing $\psi_0$ as a character of $G_\mathfrak{c}$, write $\psi_0=\bar{\psi}_{0}\cdot\psi_{\Gamma}$ according to the factorization $(\ref{eq:factor-G})$, with $\bar{\psi}_{0}$ (resp. $\psi_{\Gamma}$) a character of $\Delta_{\mathfrak{c}}$ (resp. $\Gamma_K$). We denote by $\mathscr{L}_{\fp}^{\rm Katz}(\psi_0)^{-,\iota}\in\Z_p^{\rm ur}\dBr{\Gamma^{\rm ac}}$ the image of $\mathscr{L}_{\fp,\mathfrak{c}}^{\rm Katz}(K)$ under the composite map
\[
\Z_p^{\rm ur}\llbracket G_{\mathfrak{c}}\rrbracket\rightarrow\Z_p^{\rm ur}\llbracket\Gamma_K\rrbracket\rightarrow\Z_p^{\rm ur}\llbracket\Gamma_K\rrbracket\rightarrow\Z_p^{\rm ur}\dBr{\Gamma^{\rm ac}}\rightarrow\Z_p^{\rm ur}\dBr{\Gamma^{\rm ac}},
\]
where the first arrow is given by the projection defined by $\bar{\psi}_0^{-1}\omega^{r_1/2}$, the second by twisting by $\psi_\Gamma^{-1}$, the third is the natural projection, and the last arrow is the involution given by $\gamma\mapsto\gamma^{-1}$ for $\gamma\in\Gamma^{\rm ac}$. In other words, $\mathscr{L}_{\fp}^{\rm Katz}(\psi_0)^{-,\iota}$ is the element of $\Z_p^{\rm ur}\dBr{\Gamma^{\rm ac}}$ defined by
\[
\mathscr{L}_{\fp}^{\rm Katz}(\psi_0)^{-,\iota}(\xi)=\mathscr{L}_{\fp,\mathfrak{c}}^{\rm Katz}(K)(\psi_0^{-1}\xi^{-1}\omega^{r_1/2})
\]
for all characters $\xi$ of $\Gamma^{\rm ac}$.
\end{definition}

Denote by $\tau_N$ the product of constants $\prod_{\ell\vert N}\tau_\ell$ appearing in Theorem~\ref{thm:hsieh}.
%and let $h_K$ be the class number of $K$.

\begin{theorem}\label{thm:factor}
%There exists an admissible function $\mathfrak f(k')$ such that
%or all weights with $k' \equiv k \pmod{2(p-1)}$
The following equality holds
\[
 L_p({\rm ad}(g_K)\otimes\psi_0)=u\cdot L_p({\rm ad}^0(g_K)\otimes\psi_0)\cdot\mathscr{L}_{\fp}^{\rm Katz}(\psi_0)^{-,\iota}\cdot\tau_N
\]
%\[ \Lp(\hf,g,g^\ast)(k')^2 \cdot\mathscr{L}_{\fp,\mathfrak{c}}^{\mathrm{Katz}}(K)(\psi_{k'}'/\psi_{k'})^2\cdot h_K^2 = u \cdot L_p(\hf, \ad^0(g))(k') \cdot L_{\fp,\mathfrak{c}}^{\mathrm{Katz}}(K)(\psi_{k'}'\mathbf{N}^{-k'/2}) \cdot\tau_N \]
where $u$ is a unit in $(\Lambda^{\rm ac})^\times$.
\end{theorem}

\begin{proof}
Let $\xi$ be a character of $\Gamma^{\rm ac}$ as in the statement of  Theorem~\ref{thm:sym} and Theorem~\ref{thm:hsieh}, hence in particular corresponding to a Hecke character, still denoted by $\xi$, of infinity type $(-n,n)$ with $n\geq l-1$. Noting that $\theta_{\psi_0\xi_n}$ has weight $2n+2$, from Hida's formula for the adjoint $L$-value (see \cite[Thm~.7.1]{hidatilouineI}) and Dirichlet's class number formula we obtain (cf. \cite[p.\,414]{JSW})
\begin{equation}\label{eq:classnumber}
\langle\theta_{\psi_0\xi_n},\theta_{\psi_0\xi_n}\rangle=
(2n+1)!\cdot D_K^2\cdot\frac{1}{2^{4n+4}\pi^{2n+3}}\cdot\frac{2\pi h_K}{w_K\sqrt{D_K}}\cdot L(\psi_0^{1-\tau}\xi^{1-\tau},1),
\end{equation}
where $w_K$ is the number of units in  $\cO_K$. Since $L(\psi_0^{1-\tau}\xi^{1-\tau},1)$ corresponds to the value at $s=0$ of the $L$-function for the Hecke character $\psi_0^{\tau-1}\xi^{\tau-1}\mathbf{N}^{-1}$ of infinity type $(2n+2,-2n)$, using (\ref{eq:classnumber}) the interpolation formula in Theorem~\ref{thm:katz} can be rewritten as
\begin{align*}
\mathscr{L}_{\fp,\mathfrak{c}}^{\rm Katz}(K)(\psi_0^{\tau-1}\xi^{\tau-1}\mathbf{N}^{-1})&=\biggl(\frac{\Omega_p}{\Omega_\infty}\biggr)^{4n+2}\cdot\frac{2^{6n+4}\pi^{4n+2}}{\sqrt{D_K}^{2n+1}}\cdot\frac{w_K}{D_K^2h_K}\\
&\quad\times\left(1-\frac{(\psi_0\xi\omega^{-n})(\fp)}{(\psi_0\xi\omega^{-n})(\bar{\fp})}\right)\left(1-\frac{(\psi_0\xi\omega^{-n})(\fp)}{p(\psi_0\xi\omega^{-n})(\bar{\fp})}\right)\cdot\langle\theta_{\psi_0\xi_n},\theta_{\psi_0\xi_n}\rangle.
\end{align*}
Thus together with Theorem~\ref{thm:hsieh} we find that
\begin{equation}\label{eq:LHS}
\begin{aligned}
&\Lp({\rm ad}(g_K)\otimes\psi_0)(\xi)^2 \cdot\mathscr{L}_{\fp,\mathfrak{c}}^{\mathrm{Katz}}(K)(\psi_0^{\tau-1}\xi^{\tau-1}\mathbf{N}^{-1})^2\cdot h_K^2\\
&=\biggl(\frac{\Omega_p}{\Omega_\infty}\biggr)^{8n+4}\cdot\frac{2^{8n+4}\pi^{4n}}{\sqrt{D_K}^{4n}}\cdot\Gamma(n,l,l)\cdot\mathcal{E}({\rm ad}(g),\psi_0\xi)^2\cdot\frac{w_K^2}{D_K^4}\cdot\tau_N\cdot L({\rm ad}(g_K)\otimes\psi_0^{-1}\xi^{-1}\omega^n,0).
\end{aligned}
\end{equation}

On the other hand, we have the factorization
\begin{equation}\label{eq:factor-L}
L({\rm ad}(g_K)\otimes\psi_0^{-1}\xi^{-1}\omega^n,0)=L({\rm ad}^0(g_K)\otimes\psi_0^{-1}\xi^{-1}\omega^n,0)\cdot L(\psi_0^{-1}\xi^{-1}\omega^n,0).
\end{equation}
The character $\psi_0^{-1}\xi^{-1}\omega^n$ has infinity type $(n+1,-n)$, and so is in the range of interpolation for $\mathscr{L}_{\fp,\mathfrak{c}}^{\rm Katz}(K)$. Thus combining Theorem~\ref{thm:sym} and Theorem~\ref{thm:katz} and using (\ref{eq:factor-L}) we find
\begin{equation}\label{eq:RHS}
\begin{aligned}
L_p({\rm ad}^0(g_K)\otimes &\psi_0)(\xi) \cdot\mathscr{L}_{\fp,\mathfrak{c}}^{\rm Katz}(K)(\psi_0^{-1}\xi^{-1}\omega^n)=\biggl(\frac{\Omega_p}{\Omega_\infty}\biggr)^{6n+3}\cdot\pi^{3n}\cdot\Gamma(n,l)\cdot\mathcal{E}({\rm ad}^0(g),\psi_0\xi)^2\\
&\times\biggl(\frac{\Omega_p}{\Omega_\infty}\biggr)^{2n+1}\cdot n!\cdot\left(\frac{2\pi}{\sqrt{D_K}}\right)^{n}\cdot(1-p^{-1}\psi_0\xi(\fp))^2\cdot L({\rm ad}(g_K)\otimes\psi_0^{-1}\xi^{-1}\omega^n,0).
\end{aligned}
\end{equation}
Comparing \eqref{eq:LHS} and \eqref{eq:RHS} we see that their ratio is given by $2^{7n+4}\cdot\sqrt{D_K}^{-3n-8}\cdot\tau_N$; since for varying $n$ the first two factors are interpolated by a unit in $(\Lambda^{\rm ac})^\times$, applying the functional equation of Theorem~\ref{thm:katz} this gives the result.
%
%
%The result follows by comparing the interpolation formulas since the interpolation regions match. We begin by noting that the product of $L$-values in the right corresponds to the $L$-value in the left by Artin formalism, and the same happens for the Euler and gamma factors (noting that those in the denominator of the triple product are cancelled with those by the anticyclotomic Katz in the left).
%
%Further, regarding the different periods involved in the equality, the factor $(\Omega_p/\Omega_{\infty})$ appears in the right with exponent $3k'-3+k'-1=4k'-4$, while in the left we have $(\Omega_p/\Omega_{\infty})^{4k-4}$. The same happens with the powers of $\pi$, which appear in the left with an exponent of $4k'-4-2k'=2k'-4$ and in the right with $3k'/2-3+k'/2-1=2k'-4$.
%
%Finally, observe that the local factors outside $p$, like the factor $D_K^{k_2/2}$ coming from the Katz, are also invertible if $(p,D_K)=1$, so they are included in the admissible function. The comparison of the remaining terms follows in the same way.
\end{proof}

%\begin{remark}
%Proceeding as in \cite[Thm. 3.9]{DLR1} we can make the factor $\mathfrak f(k')$ explicit. As in loc.\,cit. it may be written as \[ \mathfrak f(k') = \frac{\mathfrak f_1(k')}{\omega(k')} \Big( \frac{\mathfrak f(\Psi_{k'})}{\mathfrak f_2(k') \mathfrak f_3(k')} \Big)^2, \] where the quotient $\mathfrak f_1(k')/\omega(k')$ is related with the choice of test vectors (see Lemma 3.6); $\mathfrak f_2(k')$ measures the discrepancy between the $L$-function of the adjoint and the Petersson norm (see Lemma 3.7); and $\mathfrak f_3(k')$ accounts for the difference between the Euler factors at primes $q \mid cD_K$ defining the various $L$-series involved in our picture (see Lemma 3.8 and the discussion before it).
%\end{remark}

\section{The Euler system}\label{sec:ES}

Let $g\in S_l(N_g,\chi_g)$ be a newform as in $\S\ref{subsec:Galrep}$, and let $\psi$ be a Hecke character of $K$ of infinity type $(1-k,0)$ for some even integer $k\geq 2$, conductor $\mathfrak{c}$ prime to $p$, and central character $\varepsilon_K$. Recall from $\S\ref{sec:Lp}$ that we assume that $(p)=\fp\overline{\fp}$ splits in $K$ and (for simplicity) that $p$ does not divide the class number of $K$.

\subsection{Modified diagonal cycles}\label{subsec:modified}

Recall that, for a positive integer $m$, we write $K[m]$ for the maximal $p$-extension inside the ring class field of $K$ of conductor $m$. Recall further that $H^1_{\rm Iw}(K[mp^\infty],T) = \varprojlim_rH^1(K[mp^r],T)$. Then, for a positive integer $m$, let
\begin{equation}\label{eq:class-Trep}
\kapinftyn\in H^1_{\Iw}(K[mp^\infty],\Trep)
\end{equation}
be the class $\kappa_{\psi,g,g^\ast,m,\infty}$ constructed in \cite[Thm.~6.5]{ACR}.
(For $m=1$, this is essentially the class $\kappa^{(3)}(\hf,g,g^\ast)$ defined in $\S\ref{section:bottomclass}$, after an application of Shapiro's lemma and twisting by the inverse of the anticyclotomic character $\xi$ in (\ref{eq:tw}) below.) Since we have a direct sum decomposition
\[
H^1_{\Iw}(K[mp^\infty],\Trep)=H^1_{\Iw}(K[mp^\infty],T)\oplus H^1_{\Iw}(K[mp^\infty],T'),
\]
we can project the class $\kapinftyn$ to each of the summands. We denote its projection to the first summand as $\kapinftyns$. For the following results, we keep the notations for Selmer groups introduced in Section \ref{subsection:Selmer}.

\begin{theorem}\label{thm:ESwithextrafactor}
%Suppose that:
%\begin{itemize}
%\item $p$ splits in $K$,
%\item $p$ does not divide the class number of $K$.
%\end{itemize}
Let $\mathcal{S}$ be the set of all squarefree products of primes $q$ split in $K$ and coprime to $pN_gN_\psi$, and assume that $H^1(K[mp^s],T)$ is torsion-free for every $m \in\mathcal{S}$ and for every $s\geq 0$. There exists a collection of classes
\[
\left\lbrace\kapinftyns\in \Sel_{\Gr}(K[mp^\infty],T)\;\colon\; m\in\mathcal{S}\right\rbrace
\]
such that whenever $m, mq\in\mathcal{S}$ with $q$ a prime, we have
\begin{equation}
{\rm cor}_{K[mq]/K[m]}(\kapinftynqs)=P_{\frk{q}}(\Vrep;{\rm Fr}_{\frk{q}}^{-1})\,\kapinftyns.\nonumber
\end{equation}
\end{theorem}
\begin{proof}
This is an immediate consequence of \cite[Thm. 6.5]{ACR} and [\emph{op.\,cit.}, Prop.~6.6].
\end{proof}

\begin{comment}
\begin{proposition}
For all $n\in\mathcal{S}$, we have $\kapinftyns\in\Sel_{\Gr}(K[np^\infty],\Tsrep)$.
\end{proposition}
\begin{proof}
This follows in the same way as \cite[Prop. 6.6]{ACR}.
\end{proof}
\end{comment}

The Euler factors appearing in the previous theorem are not the ones that we want. Indeed, let $q = \frk{q} \bar{\frk{q}}$ be a prime which splits in $K$. Then
\[
P_{\frk{q}}(V_{\ad(g)}^\psi;X)=\left(1-\frac{\psi(\frk{q})X}{q^{k/2}}\right)P_{\frk{q}}(V;X),
\]
so there is an unwanted extra factor. We now deal with this problem.

%{\color{blue} R: Si fem servir la maquin\`aria de JNS potser ens podem estalviar aquest lemma.}
%We recall the following result from \cite{LZ}.
%\begin{lemma}
%For any prime $q \nmid pN_g N_{\psi}$, any Dirichlet character $\eta$ of $p$-power conductor and any integer $k \geq 0$, the map $\Frob_{q}^{p^k}-1$ is injective on $T(\eta^{-1})$.
%\end{lemma}

\begin{definition}\label{def:P'}
Let $\mathcal P'$ be the set of primes $q \nmid pN_g N_{\psi}$ split in $K$ such that
\begin{itemize}
\item $q = 1$ modulo $p$,
\item $T/(\Frob_{q}-1)T$ is a cyclic $\mathbb Z_p$-module,
\item $\Frob_{q}-1$ is bijective on $T'$.
\end{itemize}
Here $\Fr_q$ denotes any arithmetic Frobenius element for $q$. Since $T^\vee(1)\simeq T^c$ and $(T')^\vee(1)\simeq (T')^c$, the definition does not depend on this choice.
%
%Let $\cl{S}'$ be the set of squarefree products of primes in $\cl{P}'$.
\end{definition}

\begin{remark}\label{rem:sigma}
Under certain conditions, we will show in Proposition~\ref{prop:existence-of-sigma} below that there exists $\sigma\in G_K$ such that if $q$ is a prime such that $\Fr_q$ is conjugate to $\sigma$ in $\Gal(K(\mu_p,\overline{T},\overline{T'})/K)$, then it belongs to $\cl{P}'$.
\end{remark}

\begin{comment}
Let $q = \mathfrak q \bar{\mathfrak q}$ be a prime which splits in $K$, and consider the Euler factor \[ \begin{aligned} P_{\mathfrak q}(X) & = \det(1 - \Frob_{\mathfrak q}^{-1} X \mid (\Trep)^{\vee}(1)) \\ & = \Big( 1 - \frac{\psi(\mathfrak q) \alpha_g^2}{q^{k/2+l-1}} X \Big) \Big( 1 - \frac{\psi(\mathfrak q) \alpha_g \beta_g}{q^{k/2+l-1}} X \Big)^2 \Big( 1 - \frac{\psi(\mathfrak q) \beta_g^2}{q^{k/2+l-1}} X \Big). \end{aligned} \]

A trivial computation shows that \[ P_{\mathfrak q}([\mathfrak q]) \psi(\bar{\mathfrak q})^2 [\bar{\mathfrak q}]^2 = P_{\bar{\mathfrak q}}([\bar{\mathfrak q}]) \psi(\mathfrak q)^2 [\mathfrak q]^2 \pmod{q-1}, \] and the square of the Euler factor appearing in the norm relations is \[ P_{\mathfrak q}([\mathfrak q]) P_{\bar{\mathfrak q}}([\bar{\mathfrak q}]) \chi_g(q)^2 \pmod{q-1}. \]
\end{comment}

\begin{theorem}\label{thm:ES-T}
%Suppose that:
%\begin{itemize}
%\item $p$ splits in $K$,
%\item $p$ does not divide the class number of $K$.
%\end{itemize}
Let $\cl{S}'$ be the set of squarefree products of primes in $\cl{P}'$, and assume that $H^1(K[mp^s],T)$ is torsion-free for every $n\in\mathcal{S}'$ and for every $s\geq 0$. There exists a collection of classes
\[
\left\lbrace\kappa_m\in \Sel_{\Gr}(K[mp^\infty],T)\;\colon\; m\in\mathcal{S}'\right\rbrace
\]
such that $\kappa_1=\kapinftyones$ and, whenever $m, mq\in\mathcal{S}'$ with $q$ a prime, we have
\begin{equation}
{\rm cor}_{K[mq]/K[m]}(\kappa_{mq})=P_{\frk{q}}(V;{\rm Fr}_{\frk{q}}^{-1})\,\kappa_m,\nonumber
\end{equation}
where $\mathfrak{q}$ is any of the primes of $K$ above $q$.
\end{theorem}

%\begin{theorem}
%There exists a collection of classes \[
%c_m' \in H_{\Iw}^1(K[mp^{\infty}], T) \] such that
%\begin{enumerate}
%\item[(i)] $c_1' = \kappa_{\psi,g,g,1}$;
%\item[(ii)] $c_m' \in H_{\Iw,\Gr,1}^1(K[mp^{\infty}],T)^{\nu}$ for all $m$;
%\item[(iii)] if $m \in \mathcal R'$ and $q$ is a prime with $qm \in \mathcal R'$, then, \[ \cor_{K[mq]/K[m]} = P_{\mathfrak q}(\Ad^0 g \otimes \psi, \Frob_{\mathfrak q}^{-1}) c_m'. \]
%\end{enumerate}
%\end{theorem}

\begin{proof}
We construct the classes $\kappa_m$ by modifying the classes $\kapinftyns$ in Theorem~\ref{thm:ESwithextrafactor} appropriately as done in the proof of \cite[Thm. 5.3.3]{LZ}.

For each $m \in \cl{S}'$, let $\Gamma_m = \Gal(K[mp^{\infty}]/K)$. For each prime $q \mid m$, let $F_q \in \mathcal S'$ denote the unique element of $\Gamma_m$ which acts trivially on $K[q]$ and maps to $\Frob_{q}$ in $\Gamma_{m/q}$. Then, the factor $1-q^{-k/2}\psi(\mathfrak q) F_{\mathfrak q}^{-1}$ is invertible in $\mathbb Z_p\dBr{\Gamma_m}$. We now take
\[ \kappa = \prod_{q \mid m} \left(1-\frac{\psi(\frk{q})}{q^{k/2}} F_{\mathfrak q}^{-1}\right)^{-1} \kapinftyns.
\]
These classes clearly satisfy the required properties.
%(since the Greenberg Selmer groups are $\mathbb Z_p[[\Gamma_m]]$-modules).
\end{proof}

\subsection{The explicit reciprocity law}

%We now explain the relation between the non-triviality of the Euler system for $\Tsrep$ constructed in Theorem~\ref{thm:ES-T} and the $p$-adic $L$-function $L_p({\rm ad}^0(g_K)\otimes\psi_0)$ of Theorem~\ref{thm:sym}.
%The link is an easy  consequence of the reciprocity law for diagonal cycles \cite{BSV,DR3} and the factorization of Theorem~\ref{thm:factor}.
Let $K_\infty$ denote the anticyclotomic $\bb{Z}_p$-extension of $K$ and let
\[
\kapinfty
%:={\rm cor}_{K[1]/K}(\kapinftyone)
\in H^1_{\rm Iw}(K_\infty,\Trep)
\]
be the image of the class $\kapinftyone$ in (\ref{eq:class-Trep})  under the corestriction map for $K[p^\infty]/K_\infty$.
By \cite[Prop.~6.6]{ACR} we have $\kapinfty\in{\rm Sel}_{\rm bal}(K_\infty,\Trep)$; in particular, the restriction ${\rm res}_{\overline{\fp}}(\kapinfty)$ lands in the image of the natural map
\[
H^1_{\rm Iw}(K_{\infty,\overline{\fp}},\mathscr{F}_{\overline{\fp}}^{\rm bal}(\Trep))\longrightarrow H^1_{\rm Iw}(K_{\infty,\overline{\fp}},\Trep)
\]
(see (\ref{eq:local-bal})). %for the definition of $\mathscr{F}_{\overline{\fp}}^{\rm bal}(\Trep)$).
Note that this map is an injection under our hypotheses. On the other hand, let
\begin{equation}\label{eq:kapinftys}
\kapinftys\in{\rm Sel}_{\Gr}(K_\infty,T)
\end{equation}
be the image of the class $\kappa_1=\kapinftyones$ of Theorem~\ref{thm:ES-T} under the corestriction map. Thus $\kapinftys$ is the projection of $\kapinfty$ onto the first direct summand in the decomposition
\[
{\rm Sel}_{\rm bal}(K_\infty,\Trep)={\rm Sel}_{\rm bal}(K_\infty,T)\oplus{\rm Sel}_{\rm bal}(K_\infty,T'),
\]
and since $\mathscr{F}_{\overline{\fp}}^{\rm bal}(\Trep)$ is contained in $T$, it is clear that
\begin{equation}\label{eq:Log-factor}
{\rm res}_{\overline{\fp}}(\kapinfty)={\rm res}_{\overline{\fp}}(\kapinftys).
\end{equation}

\begin{definition}\label{def:L-psi}
Put $\psi_0=\psi\lambda^{2-k}$ as in Remark~\ref{rem:psi_0},
%which is a Hecke character of infinity type $(-1,0)$, conductor $\mathfrak{c}\fp^e$ with $e\in\{0,1\}$ and the
and define
\begin{equation}\label{eq:tw}
L_p({\rm ad}^0(g_K)\otimes\psi):={\rm Tw}_{\xi}\bigl(L_p({\rm ad}^0(g_K)\otimes\psi_0)\bigr),
%\quad\mathscr{L}_{\fp}^{\rm Katz}(\psi)^{-,\iota}:={\rm Tw}_\xi(\mathscr{L}_{\fp}^{\rm Katz}(\psi)^{-,\iota}),
\end{equation}
where ${\rm Tw}_\xi:\Lambda^{\rm ac}\rightarrow\Lambda^{\rm ac}$ is the twisting homomorphism %$\gamma\mapsto\xi(\gamma)\gamma
for the character $\xi:=(\lambda^{1-\tau})^{k/2-1}$. Similarly, define $\mathscr{L}_p({\rm ad}(g_K)\otimes\psi)$, $\mathscr{L}_{\fp,\psi}^{\rm Katz}(K)^-$, $L_p({\rm ad}(g_K)\otimes\psi)$, and $\mathscr{L}_{\fp}^{\rm Katz}(\psi)^{-,\iota}$ by twisting the corresponding $p$-adic $L$-functions defined for $\psi_0$ in $\S\ref{subsec:factor}$.
\end{definition}

For the statement of the next result, note that $\psi_0$ has the same restriction to $\Delta_\mathfrak{c}$ as $\psi$.
%see (\ref{eq:factor-G}).

%It follows from \cite[Cor.~8.2]{BSV} that the image of $\kappa_1=\kapinftyones$ under the restriction map
%\[
%{\rm res}_{\bar{\fp}}:H^1_{\rm Iw}(K_\infty,\Trep)\longrightarrow H^1_{\rm Iw}(K_{\infty,\bar{\fp}},\Trep)
%\]
%is contained in
%\[
%H^1_{\rm Iw,\rm bal}(K_{\infty,\bar{\fp}},\Trep):={\rm im}\bigl(H^1_{\rm Iw}(K_{\infty,\bar{\fp}},\mathcal{F}_{\bar\fp}^{\rm bal}(\Trep))\longrightarrow H^1_{\rm Iw}(K_{\infty,\bar{\fp}},\Trep)\bigr),
%\]
%where $\mathcal{F}_{\bar\fp}^{\rm bal}(\Trep)=T_g^+\otimes T_{g^*}^+(\psi_\mathfrak{P}^{-1})(1-k/2)$.

%It is easy to see that the map in this definition is an injection, and in the following we shall use this to view ${\rm res}_{\bar\fp}(\kapinfty)$ as a class in $H^1_{\rm Iw}(K_{\infty,\bar{\fp}},\mathcal{F}_{\bar\fp}^{\rm bal}(\Trep))$. Note also that
%\[
%\mathcal{F}_{\bar\fp}^{\rm bal}(\Tsrep):=\mathcal{F}_{\bar\fp}^{\rm bal}(\Trep)\cap T=\mathcal{F}_{\bar\fp}^{\rm bal}(\Trep),
%\]
%and so ${\rm res}_{\bar\fp}(\kapinfty)={\rm res}_{\bar\fp}(\kapinftys)$.

\begin{theorem}\label{thm:ERL}
There exists an injective $\Lambda^{\rm ac}$-module map with pseudo-null cokernel
\[
\mathfrak{Log}:H^1_{\rm Iw}(K_{\infty,\overline{\fp}},\mathscr{F}_{\overline{\fp}}^{\rm bal}(\Trep))\longrightarrow\Z_p^{\rm ur}\dBr{\Gamma^{\rm ac}}
\]
such that
\begin{align*}
\mathfrak{Log}({\rm res}_{\bar{\fp}}(\kapinftys))
&=h_K\cdot\mathscr{L}_{\fp,{\psi}}^{\rm Katz}(K)^-\cdot\mathscr{L}_p({\rm ad}(g_K)\otimes\psi).
%&=u\cdot L_p({\rm ad}^0(g_K)\otimes\psi_0)\cdot\mathscr{L}_{\fp}^{\rm Katz}(\psi_0)^{-,\iota}\cdot\tau_N,
\end{align*}
%where $u\in(\Lambda^{\rm ac})^\times$ and $\tau_N\in\Q^\times$ are as in Theorem~\ref{thm:factor}.
\end{theorem}

\begin{proof}
%This is a reformulation of  using the factorization of Theorem~\ref{thm:factor}.
Let $\mathbb{V}(\hf,g,g^*)$ and $\mathbb{V}(\hf,g,g^*)_{f}$ be as in $\S\ref{section:bottomclass}$ (corresponding to $\mathbb{V}_{\hf gg^*}^\dagger$ and $\mathbb{V}_{\hf}^{gg^*}$, respectively, in the notation of \cite[\S{8.2}]{ACR}). Then, identifying $G_{\bb{Q}_p}$ with $G_{K_{\overline{\frk{p}}}}$ via the composition of the embedding $\iota_p:\overline{\Q}\hookrightarrow\overline{\Q}_p$ fixed in the introduction with complex conjugation, we get an isomorphism of $\Lambda_\hf[G_{K_{\overline{\frk{p}}}}]$-modules
\begin{equation}\label{eq:gh_f}
\mathbb{V}(\hf,g,g^*)_{f}=\mathbb{V}_{\hf}^-\hat\otimes_{\mathcal{O}}T_g^+\otimes T_{g^*}^+(\epsilon_{\rm cyc}^{1-l}\bm{\kappa}_{f}^{-1/2})\simeq\mathscr{F}_{\overline{\fp}}^{\rm bal}(\Tsrep)\otimes\xi\Psi_S^{(\tau-1)/2}.
\end{equation}
By \cite[Thm.~7.4]{ACR}, after extending scalars to $\Z_p^{\rm ur}$, the composition of the $\Lambda_{\hf}$-linear map
\[
H^1(\Q_p,\mathbb{V}(\hf,g,g^*)_f)\longrightarrow\Lambda_{\hf}
\]
of [\emph{op.\,cit.}, Prop.~7.3] with the isomorphism $\Z_p^{\rm ur}\hat\otimes\Lambda_{\hf}\simeq\Z_p^{\rm ur}\dBr{\Gamma^{\rm ac}}$ given by $\gamma\mapsto\gamma^{\tau-1}$ sends the class
$\kappa^{(3)}(\hf,g,g^*)$ recalled in $\S\ref{section:bottomclass}$ to the product
\[
h_K\cdot\mathscr{L}_{\fp,\psi_0}^{\rm Katz}(K)^-\cdot\mathscr{L}_p({\rm ad}(g_K)\otimes\psi_0),
\]
noting that by Proposition~\ref{prop:congr} the first two factors generate the congruence ideal of $\hf$. Taking twists by $\xi$ and using the isomorphism
\[
H^1_{\rm Iw}(K_{\infty,\overline{\frk{p}}},\mathscr{F}_{\overline{\fp}}^{\rm bal}(\Tsrep)\otimes\xi)\simeq H^1(\Q_p,\mathbb{V}(\hf,g,g^*)_{f})
\]
induced by (\ref{eq:gh_f}) and using (\ref{eq:Log-factor}), the result follows.
\end{proof}

\begin{corollary}\label{cor:ERL}
The map $\mathfrak{Log}$ of Theorem~\ref{thm:ERL} satisfies
\[
\bigl(\mathfrak{Log}({\rm res}_{\overline{\fp}}(\kapinftys))^2\bigr)=\bigl(L_p({\rm ad}^0(g_K)\otimes\psi)\cdot\mathscr{L}_{\fp}^{\rm Katz}(\psi)^{-,\iota}\bigr)
\]
as ideals in $\Z_p^{\rm ur}\dBr{\Gamma^{\rm ac}}\otimes\Q_p$.
\end{corollary}

\begin{proof}
This is clear from Theorem~\ref{thm:ERL} and the factorization in Theorem~\ref{thm:factor}.
\end{proof}

\section{Verifying the hypotheses}\label{sec:verify}

%Let $K/\bb{Q}$ be an imaginary quadratic field of discriminant $-D$. Let $\psi$ be a Hecke character of $K$ of infinity type $(1-k,0)$ for some positive even integer $k\geq 2$, conductor $\frk{f}$ and central character $\varepsilon_K$. Let $g$ be a non-Eisenstein modular form of weight $l$ and character $\chi_g$. Throughout this section, we make the further assumption that $g$ is not of CM type.

Let $g$ and $\psi$ be as introduced in \S\ref{subsec:Galrep} and \S\ref{subsec:adjoint}, respectively. Recall that, given a rational prime $p$ and a sufficiently large finite extension $E/\bb{Q}_p$, we define the $G_K$-representation
\[
V={\rm ad}^0(V_g)(\psi_\mathfrak{P}^{-1})(1-k/2),
\]
where $\rho_g:G_\Q\rightarrow{\rm Aut}_E(V_g)\simeq{\rm GL}_2(E)$ is the $p$-adic Galois representation attached to $g$. The aim of this section is to give conditions under which the hypotheses in the general results of \cite{JNS} are satisfied for $V$. Let $K(p^\infty)^\circ$ denote the maximal abelian extension of $K$ unramified at primes not dividing $p$. Then, the crucial condition that we need to verify is the existence of an element $\sigma \in \Gal(\bar K/K(p^{\infty})^{\circ})$ such that $T/(\sigma-1)T$ is a free $\mathcal O$-module of rank one, where $\cl{O}$ is the ring of integers of $E$.
\begin{comment}
As shown in \cite[Prop. 4.1.1]{Loe}, such an element does not exist if we replace $T$ by $T_g \otimes T_{g^*}(\psi^{-1})(1-k/2)$, but we will check that under mild assumptions we can get this condition for the adjoint (three-dimensional) representation $T$.
\end{comment}

As in \cite[\S3.1]{Loe} we define an open subgroup $H_g\subseteq G_{\bb{Q}}$, a quaternion algebra $B_g$ and an algebraic group $G_g$. Let $H=H_g\cap G_{K(\frk{c})^\circ}$. Then we have an adelic representation
\[
\tilde{\rho}_{g}:H\longrightarrow G_g(\hat{\bb{Q}})
\]
and representations
\[
\tilde{\rho}_{g,p}: H\longrightarrow G_g(\bb{Q}_p)
\]
for every rational prime $p$, and, according to \cite[Thm. 2.2.2]{Loe}, for all but finitely many $p$ we can conjugate $\tilde{\rho}_{g,p}$ so that $\tilde{\rho}_{g,p}(H)=G_g(\bb{Z}_p)$.

Let $L$ be a finite extension of $K$ containing the Fourier coefficients of $g$ and the image of the Hecke character $\psi$. Let $\frk{P}$ be a prime of $L$ above some rational prime $p$, and let $E=L_\frk{P}$.

\begin{definition}\label{def:good}
We say that the prime $\frk{P}$ is \textit{good} if the following conditions hold:
\begin{itemize}
\item $p\geq 3$;
\item $p$ is unramified in $B_g$;
\item $p$ is coprime to $\frk{c}$ and $N_g$;
\item $\tilde{\rho}_{g,p}(H)=G_g(\bb{Z}_p)$;
\item $E=\bb{Q}_p$.
\end{itemize}
\end{definition}

\begin{lemma}
Assume that there is at least one prime which divides $D_K$ but not $N_g$. Then, if $\frk{P}$ is a good prime,
\[
\rho_{g,\frk{P}}(H\cap G_{K(p^\infty)^\circ})=\mathrm{SL}_2(\bb{Z}_p).
\]
\end{lemma}
\begin{proof}
The proof of this result is very similar to the proof of \cite[Lem. 8.9]{ACR}. We include it here for the convenience of the reader.

Let $\bb{Q}(\rho_g)$ be the Galois extension of $\bb{Q}$ cut out by the representations $\rho_g$. It is unramified outside $pN_g$. Therefore, the condition on $D_K$ implies that $K\cap \bb{Q}(\rho_g)=\bb{Q}$. Moreover, since any Galois extension of $\bb{Q}$ contained in the anticyclotomic $\bb{Z}_p$-extension $K_\infty$ of $K$ must itself contain $K$, we also have $K_\infty\cap\bb{Q}(\rho_g)=\bb{Q}$.

The conditions on $\frk{P}$ imply that
\[
\rho_{g,\frk{P}}(H\cap G_{\bb{Q}(\mu_{p^\infty})})=\mathrm{SL}_2(\bb{Z}_p),
\]
and, from the remarks in the previous paragraph, it follows that
\[
\rho_{g,\frk{P}}(H\cap G_{K_\infty(\mu_{p^\infty})})=\mathrm{SL}_2(\bb{Z}_p).
\]
Finally, since $H\cap G_{K(p^\infty)^\circ}$ is a normal subgroup of $H\cap G_{K_\infty(\mu_{p^\infty})}$ of index dividing $p-1$ and there are no such subgroups in $\mathrm{SL}_2(\bb{Z}_p)$, the lemma follows.
\end{proof}

Now fix a good prime $\frk{P}$ and define $\bb{Z}_p[G_K]$-modules $T=\ad^0 T_g(\psi_{\frk{P}}^{-1})(1-k/2)$ and $T'= \bb{Z}_p(\psi_{\frk{P}}^{-1})(1-k/2)$. Let $V=T\otimes \bb{Q}_p$ and $V'=T'\otimes \bb{Q}_p$.

\begin{proposition}\label{prop:existence-of-sigma}
Assume that there is at least one prime which divides $D_K$ but not $N_g$. Suppose that there exists $\eta\in G_{K(p^\infty)^\circ}$ such that $\chi_g(\eta)\psi_\frk{P}(\eta)$ is a square modulo $p$ and $\psi_{\frk{P}}(\eta)^2\neq 1$ modulo $p$. Then there exists $\sigma\in G_{K(p^\infty)^\circ}$ such that
\begin{itemize}
\item $T/(\sigma-1)T$ is free of rank 1 over $\bb{Z}_p$,
\item $\sigma-1$ acts invertibly on $T'$.
\end{itemize}
\end{proposition}
\begin{proof}
We closely follow the proof of \cite[Prop. 5.2.1]{LZ} (see also the proof of \cite[Lem 5.10]{ACR} and \cite[Prop. 4.2.1]{Loe}). By the previous lemma the image of $\eta H\cap G_{K(p^\infty)^\circ}$ under $\rho_{g,\frk{P}}$ contains all the elements of the form
\[
\begin{pmatrix} x & 0 \\ 0 & x^{-1}\chi_g(\eta)\end{pmatrix},\quad x\in \bb{Z}_p^\times.
\]
Choose $x$ such that $x^2=\chi_g(\eta)\psi_{\frk{P}}(\eta)$. Choose $\sigma\in \eta H\cap G_{K(p^\infty)^\circ}$ whose image under $\rho_{g,\frk{P}}$ is given by the element above, with the choice of $x$ which we have just specified. Then, the eigenvalues of $\sigma$ acting on $T$ are $1$, $\psi_{\frk{P}}^{-1}(\eta)$ and $\psi_{\frk{P}}^{-2}(\eta)$ and the eigenvalue of $\sigma$ acting on $T'$ is $\psi_{\frk{P}}^{-1}(\eta)$. The result follows from the assumptions on $\eta$.
\end{proof}

\section{Applications}\label{sec:applications}

Let $g\in S_l(N_g,\chi_g)$ and $\psi$ be a Hecke character of $K$ of infinity type $(1-k,0)$ for some even integer $k\geq 2$ as introduced in $\S\ref{sec:Selmer}$, and recall that we consider the $E$-valued $G_K$-representation $V$ in Definition~\ref{def:V}.
%\[
%V:=\ad^0(V_g)(\psi_{\frk{P}}^{-1})(1-k/2),
%\]
%where $\ad^0(V_g)\subset{\rm End}_E(V_g)$ is the adjoint representation on the trace $0$ endomorphisms of $V_g$.
We begin by collecting a set of hypotheses for our later reference.

\begin{hypotheses}\label{hyp:h1-h6}
\hfill
\begin{itemize}
\item[{\rm (h1)}] $p$ splits in $K$,
\item[{\rm (h2)}] $p\nmid h_K$,
\item[{\rm (h3)}] the conditions in Proposition~\ref{prop:congr} hold,
\item[{\rm (h4)}] $g$ is ordinary at $p$ and non-Eisenstein mod $p$,
\item[{\rm (h5)}] $g$ is not of CM type,
\item[{\rm (h6)}] $\mathfrak{P}$ is a good prime in the sense of Definition~\ref{def:good},
\item[{\rm (h7)}] the conditions in Proposition~\ref{prop:existence-of-sigma} hold.
\end{itemize}
\end{hypotheses}

\subsection{The Bloch--Kato conjecture}

We begin with a standard lemma, whose proof follows from the same argument as in \cite[Lem.~9.1]{ACR}.

\begin{lemma}\label{lem:BK}
The Bloch--Kato Selmer group of $V$ is given by
\[
\Sel(K,V)\simeq
\begin{cases}
\Sel_{\rm bal}(K,V) &\textrm{if $2\leq k < 2l$,}\\[0.2em]
\Sel_{\rm unb}(K,V) &\textrm{if $k \geq 2l$.}
\end{cases}
\]
\end{lemma}

Let $\kapinftys$ be as in (\ref{eq:kapinftys}), and denote by
\[
\kaps\in{\rm Sel}_{\rm bal}(K,T)
\]
the image of $\kapinftys$ under the corestriction $H^1_{\rm Iw}(K_\infty,T)\rightarrow H^1(K,T)$.

\begin{theorem}\label{thm:A}
Assume hypotheses (h1)--(h7). Then the following implication holds:
\[
\kaps \neq 0\quad\Longrightarrow\quad{\rm dim}_E\,\Sel_{\Gr}(K,V)=1.
\]
In particular, if $2\leq k < 2l$ and $\kaps \neq 0$ then the Bloch--Kato Selmer group $\Sel(K,V)$ is one-dimensional.
\end{theorem}

\begin{proof}
This follows from the general theory of anticyclotomic Euler systems developed in \cite{JNS} (see \cite[\S{8}]{ACR} for a summary) applied to the Euler system constructed in Theorem~\ref{thm:ES-T}. By Proposition~\ref{prop:existence-of-sigma}, Hypotheses~\ref{hyp:h1-h6} give sufficient conditions for the general results of \cite{JNS} to apply in our case. Note also that for the application of these results it suffices to have an anticyclotomic Euler system consisting of classes indexed by squarefree products of primes $q$ in a positive density set $\mathcal{P}'$ of primes split in $K$, as is the case for the anticyclotomic Euler system of Theorem~\ref{thm:ES-T} (see Remark \ref{rem:sigma}).
\end{proof}

%Recall that $\theta_\psi\in S_k(\Gamma_0(N_\psi))$ is the theta series attached to $\psi$, and we have the decomposition
%\[
%V_{g,g^*}^\psi=V\oplus V',
%\]
%where $g\in S_l(N_g,\chi_g)$, $V={\rm ad}^0(V_g)(\psi_\mathfrak{P}%^{-1})(-n)$ and $V'=E(\psi_\mathfrak{P}^{-1})(-n)$.

%As usual, put $N=\lcm(N_g,N_\psi)$.

%{\color{blue} falta fer retocs per aqui  d'acord amb la definicio (\ref{eq:tw})}

Theorem~\ref{thm:A} can be viewed as a result towards the Bloch--Kato conjecture for $V$ in rank $1$. The next result establishes cases of the same conjecture in rank $0$.

\begin{theorem}\label{thm:thmC}
Assume hypotheses (h1)--(h7), and in addition that:
\begin{itemize}
\item $\varepsilon_{\ell}(V_{fgg^\ast}) = +1$ for all primes $\ell\mid N$,
\item ${\rm gcd}(N_g,N_\psi)$ is squarefree,
\item $L(\theta_\psi,k/2) \neq 0$.
\end{itemize}
If $k\geq 2l$ then the following implication holds:
\[
L(V,0)\neq 0\quad\Longrightarrow\quad \Sel(K,V)=0.
\]
\end{theorem}

\begin{proof}
By Theorem~\ref{thm:sym} and Definition~\ref{def:L-psi} we see that
\[
%L(V,0)\doteq L_p({\rm ad}^0(g_K)\otimes\psi)(\xi_{0}),
L(V,0)\neq 0\quad\Longrightarrow\quad L_p({\rm ad}^0(g_K)\otimes\psi)(\xi_{\rm triv})\neq 0,
\]
where
%$\doteq$ stands for equality up to a nonzero factor and $\xi_{0}$
$\xi_{\rm triv}$ is the trivial character of $\Gamma^{\rm ac}$. Similarly, from Theorem~\ref{thm:katz} and Definition~\ref{def:L-psi} we see that
\[
%L(\theta_\psi,k/2)%=L(\psi_0^{-1}\xi^{-1},0)
%\doteq\mathscr{L}_\fp^{\rm Katz}(\psi)^{-,\iota}(\xi_{0}).
L(\theta_\psi,k/2)\neq 0\quad\Longrightarrow\quad\mathscr{L}_\fp^{\rm Katz}(\psi)^{-,\iota}(\xi_{\rm triv})\neq 0.
\]

Therefore by the factorization in Theorem~\ref{thm:factor} we thus see that $L_p({\rm ad}(g_K)\otimes\psi)(\xi_{\rm triv})\neq 0$, and so $\kappa_{\psi,{\rm ad}^0(g)}\neq 0$ by the explicit reciprocity law of Corollary~\ref{cor:ERL}. The result now follows from Theorem~\ref{thm:A} and global duality by the same argument as in \cite[Thm.~9.5]{ACR}.
\end{proof}

\begin{remark}
The hypotheses in Theorem~\ref{thm:thmC} and the decomposition (\ref{eq:dec-V}) imply that the sign of the functional equation for $L(V,s)$ is $+1$, and so the nonvanishing of $L(V,0)$ is expected to hold generically.
%The root number conditions in Theorem~\ref{thm:thmB} imply that the sign in the functional equation of $L(V,s)$ is $+1$, so we expect the condition $L(V,0)\neq 0$ to hold generically.
\end{remark}

%\begin{remark}
%Write $\psi=\psi_0\lambda^{2n}$ as in the proof of Theorem~\ref{thm:thmB}, and denote this by $\psi^{(n)}$. Since for $n\equiv 0\pmod{p-1}$ the character $(\lambda/\bar{\lambda})^{n}$ is unramified, the root number condition on $\theta_{\psi^{(2n+2)}}$ for any fixed $n \equiv 0\pmod{p-1}$ (e.g. for $\theta_{\psi_0}$) implies that
%\[
%L(\theta_{\psi^{(k)}},n+1)\neq 0
%\]
%for all but finitely many $n\geq 0$ with $n\equiv 0\pmod{p-1}$ by Greenberg's nonvanishing results
%\cite[Thm.~4]{greenberg-BDP}, \cite[Thm.~1]{greenberg-critical}
%\cite{greenberg-BDP,greenberg-critical}
%(cf. \cite[\S{3.11}]{deshalit} and \cite{rohrlich-ac}).
%\end{remark}

%\begin{remark}
%Proceeding as in \cite[Thm.~9.6]{ACR}, we can also bound the length of ${\rm Sel}(K,A)$ in terms of the $p$-adic valuation of an integral normalization of $L(V,0)$. We leave the details to the interested reader.
%\end{remark}

\subsection{The Iwasawa main conjecture}

Here we deduce our main result towards the anticyclotomic Iwasawa main conjecture for $V$.

Since $\psi$ has central character $\varepsilon_K$ by assumption, its associated theta series $\theta_\psi$ has trivial nebentypus. In the following we denote by $\varepsilon(\theta_\psi)$ its global root number.

%{\color{blue} Estava explorant una mica com podrien anar les coses en base a la factoritzacio que has escrit, suposo que queda una Katz a cada costat, oi? Posa-ho com vulguis que queda millor}

%Let $\mathscr{L}_{\fp,\mathfrak{c}}^{\mathrm{Katz}}(K)(\psi'/\psi)$ (resp. $\mathscr{L}_{\fp,\mathfrak{c}}^{\mathrm{Katz}} (K)(\psi' \mathbf{N}^{-k/2})$)  stand for the Katz $p$-adic $L$-function whose specialization at weight $k'$ is $\mathscr{L}_{\fp,\mathfrak{c}}^{\mathrm{Katz}} (K)(\psi_{k'}'/\psi_{k'})$ (resp. $\mathscr{L}_{\fp,\mathfrak{c}}^{\mathrm{Katz}} (K)(\psi_{k'}' \mathbf{N}^{-k'/2})$).

\begin{theorem}\label{thm:thmD}
Assume hypotheses (h1)--(h7), and in addition that:
\begin{itemize}
\item $\varepsilon_\ell(V_{fgg^\ast})=+1$ for all primes $\ell\mid N$,
\item $\varepsilon(\theta_\psi)=+1$,
\item ${\rm gcd}(N_g,N_\psi)$ is squarefree.
\end{itemize}
If the $p$-adic $L$-function $L_p({\rm ad}^0(g_K)\otimes\psi)$ is nonzero, then the Pontryagin dual of ${\rm Sel}_{\rm unb}(K_\infty,A)$ is $\Lambda^{\rm ac}$-torsion, with
\[
\Char_{\Lambda^{\rm ac}}\bigl({\rm Sel}_{\rm unb}(K_\infty,A)^\vee\bigr)\supset\bigl(L_p({\rm ad}^0(g_K)\otimes\psi)\cdot\mathscr{L}_{\fp}^{\rm Katz}(\psi)^{-,\iota}\bigr)
\]
in $\Z_p^{\rm ur}\dBr{\Gamma^{\rm ac}}\otimes_{\mathbb Z_p}\mathbb Q_p$.
%\end{enumerate}
\end{theorem}

\begin{proof}
The assumption that $\varepsilon(\theta_\psi)=+1$ implies that the anticyclotomic projection $\mathscr{L}_\fp^{\rm Katz}(\psi)^{-,\iota}$ is nonzero by Greenberg's nonvanishing results \cite{greenberg-BDP}. Since $L_p({\rm ad}^0(g_K)\otimes\psi)\neq 0$ by hypothesis, together with the factorization in Theorem~\ref{thm:factor} it follows
%(since we also assume $L_p({\rm ad}^0(g_K)\otimes\psi_0)\neq 0$)
that
\[
L_p({\rm ad}(g_K)\otimes\psi)\neq 0.
\]
By Corollary~\ref{cor:ERL}, this shows that the class $\kapinftys$ is non-torsion.
%In light of the decomposition
%\[
%{\rm Sel}_{\rm bal}(K_\infty,T_{g,g^*}^\psi)={\rm Sel}_{\rm bal}(K_\infty,T)\oplus{\rm Sel}(K_\infty,T')
%\]
%and the fact that ${\rm Sel}(K_\infty,T')=0$ by our hypothesis that $\psi_0$ has root number $+1$ (see \cite[Thm.~2.4.17]{AH-ord} in the case where $\psi_0$ correspond to $E/\Q$ and  \cite[Thm.~2.1]{arnold} for the general case), it follows that the projection $\kappa_{g,\infty}^\psi$ of $\kappa_{g,g^*,\infty}^\psi$ to ${\rm Sel}_{\rm bal}(K_\infty,T)$ is not $\Lambda^{\rm ac}$-torsion.
By the general results of \cite{JNS} (see also \cite[Thm.~8.5]{ACR}), we thus conclude that $X_{\rm bal}(K_\infty,A)$ and ${\rm Sel}_{\rm bal}(K_\infty,T)$ both have $\Lambda^{\rm ac}$-rank one, with
\[
\Char_{\Lambda^{\rm ac}}(X_{\rm bal}\bigl(K_\infty,A)_{\tors}\bigr) \supset
\Char_{\Lambda^{\rm ac}}\biggl(\frac{\Sel_{\rm bal}(K_\infty,T)}{\Lambda^{\rm ac} \cdot\kapinftys}\biggr)^2.
\]
The result now follows from this by the same argument as in the proof of \cite[Thm.~7.15]{ACR} based on Poitou--Tate duality and the explicit reciprocity law of Corollary~\ref{cor:ERL}.
\end{proof}

\bibliographystyle{amsalpha}
\bibliography{Symsquare-refs}

\end{document}